\documentclass[12pt]{amsart}
\usepackage{amsmath}
\usepackage{amssymb}
\usepackage{euscript}
\usepackage{array}
\usepackage[all]{xy}
\usepackage{longtable}

\newtheorem{theorem}{Theorem}[subsection]
\newtheorem{lemma}[theorem]{Lemma}
\newtheorem{proposition}[theorem]{Proposition}
\newtheorem{corollary}[theorem]{Corollary}
\newtheorem{conjecture}[theorem]{Conjecture}

\theoremstyle{definition}
\newtheorem*{definition}{Definition}

\theoremstyle{remark}
\newtheorem{remark}[theorem]{Remark}

\newcommand{\FF}{\mathbb{F}}
\newcommand{\ZZ}{\mathbb{Z}}
\newcommand{\QQ}{\mathbb{Q}}

\newcommand{\LL}{\mathbb{L}}
\newcommand{\TT}{\mathbb{T}}
\newcommand{\GG}{\mathbb{G}}
\newcommand{\EE}{\mathbb{E}}
\newcommand{\CC}{\mathbb{C}}

\newcommand{\bg}{\mathbf{g}}
\newcommand{\bh}{\mathbf{h}}
\newcommand{\bm}{\mathbf{m}}
\newcommand{\bn}{\mathbf{n}}
\newcommand{\bp}{\mathbf{p}}
\newcommand{\bX}{\mathbf{X}}

\newcommand{\cA}{\mathcal{A}}

\newcommand{\cM}{\mathcal{M}}
\newcommand{\cN}{\mathcal{N}}

\newcommand{\cP}{\mathcal{P}}
\newcommand{\cT}{\mathcal{T}}

\DeclareMathOperator{\Lie}{Lie}
\DeclareMathOperator{\Ker}{Ker}

\DeclareMathOperator{\GL}{GL}
\DeclareMathOperator{\Mat}{Mat}
\DeclareMathOperator{\Cent}{Cent}
\DeclareMathOperator{\End}{End}

\DeclareMathOperator{\Spec}{Spec}

\DeclareMathOperator{\Hom}{Hom}
\DeclareMathOperator{\Res}{Res}

\DeclareMathOperator{\Id}{Id}
\DeclareMathOperator{\im}{Im}
\DeclareMathOperator{\trdeg}{tr.deg}

\newcommand{\ok}{\overline{k}}
\newcommand{\tpi}{\widetilde{\pi}}
\newcommand{\oFqt}{\overline{\FF_q(t)}}
\newcommand{\iso}{\stackrel{\sim}{\to}}

\begin{document}
\title[Periods and logarithms of Drinfeld modules]{Algebraic
relations among periods and logarithms of rank~$2$ Drinfeld modules}
\author{Chieh-Yu Chang}
\address{National Center for Theoretical Sciences, Mathematics Division,
National Tsing Hua University, Hsinchu City 30042, Taiwan
  R.O.C.}
\address{Department of Mathematics, National Central
  University, Chung-Li 32054, Taiwan R.O.C.}

\email{cychang@math.cts.nthu.edu.tw}

\author{Matthew A. Papanikolas}
\address{Department of Mathematics, Texas A{\&}M University, College Station,
TX 77843-3368, USA} \email{map@math.tamu.edu}

\thanks{The second author was supported by NSF Grant DMS-0600826.}


\date{July 19, 2008}

\begin{abstract}
For any rank $2$ Drinfeld module $\rho$ defined over an algebraic
function field, we consider its period matrix $P_{\rho}$, which is
analogous to the period matrix of an elliptic curve defined over a
number field. Suppose that the characteristic of $\FF_q$ is odd and
$\rho$ is without complex multiplication. We show that the
transcendence degree of the field generated by the entries of
$P_{\rho}$ over $\FF_q(\theta)$ is $4$. As a consequence, we show
also the algebraic independence of Drinfeld logarithms of algebraic
functions which are linearly independent over $ \FF_q(\theta)$.
\end{abstract}


\maketitle

\section{Introduction}

This paper focuses on the algebraic independence of periods,
quasi-periods, and logarithms of Drinfeld modules of rank~$2$,
without complex multiplication, that are defined over the algebraic
closure of $\FF_q(\theta)$. By the main theorem of
\cite{Papanikolas}, which itself builds on results in \cite{ABP},
the proofs break into two parts.  The first is to relate the
quantities in question to special values of solutions of certain
Frobenius difference equations.  For periods and quasi-periods these
problems are solved using Anderson generating functions together
with methods inspired by \cite{Pellarin}.  For logarithms we solve
difference equations related to extensions of the trivial $t$-motive
by the $t$-motive associated to the Drinfeld module. The second part
is to show that the Galois group of the system of difference
equations has maximal dimension. The difficulty here is that these
quantities can have many linear relations, and characterizing these
linear relations in terms of the dimension of the Galois group is
quite complicated.

Yu \cite{Yu90,Yu97} proved that all linear relations among periods
and logarithms of Drinfeld modules of arbitrary rank are the ones
that are expected, ultimately relying on the Sub-$t$-module Theorem
of \cite{Yu97}.  Using Yu's methods, Brownawell \cite{Brownawell}
extended these results to values of quasi-periodic functions.  Later
David and Denis \cite{DavidDenis} used Yu's theorem to show there
are no quadratic relations among periods of rank~$2$ Drinfeld
modules without complex multiplication.

\subsection{Periods and Quasi-periods of Drinfeld modules}
Let $\FF_q$ be the finite field of $q$ elements, where $q$ is a
power of a prime $p$. Let $k:=\FF_q(\theta)$ be the rational
function field in a variable $\theta$ over $\FF_q$. Let
$\CC_{\infty}$ be the completion of an algebraic closure of
$\FF_q((\frac{1}{\theta})) $ with respect to the non-archimedean
absolute value of $k$ for which $|\theta|_{\infty}=q$. Let
$\CC_{\infty}[\tau]$ be the twisted polynomial ring in $\tau$ over
$\CC_{\infty}$ subject to the relation $\tau c=c^{q} \tau$ for $c\in
\CC_{\infty}$.

Now let $t$ be an independent variable from $\theta$. A Drinfeld
$\FF_q[t]$-module $\rho$ (with generic characteristic) is an
$\FF_q$-algebra homomorphism $\rho:\FF_q[t]\rightarrow
\CC_{\infty}[\tau]$ such that for all $a\in \FF_q[t]$, the constant
term of $\rho_a$, as a polynomial in $\tau$, is $a(\theta)$.  Given
a subfield $L$ of $\CC_{\infty}$ containing~$k$, we say that $\rho$
is defined over $L$ if all the coefficients of $\rho_{t}$ fall in
$L$. The degree of $\rho_{t}$ in $\tau$ is called the rank of
$\rho$.

The period lattice of a rank $r$ Drinfeld $\FF_q[t]$-module $\rho$
is defined as follows.  The exponential function of $\rho$ is the
function $\exp_{\rho}(z)=z+\sum_{i=1}^{\infty} \alpha_{i}z^{q^{i}}$,
with $\alpha_{i}\in \CC_{\infty}$, such that $\exp_{\rho}(\theta
z)=\rho_{t}(\exp_{\rho}(z))$.  It can be shown that $\exp_\rho$ is
entire, $\FF_q$-linear, and surjective.  Let $\Lambda_{\rho}$ be the
kernel of $\exp_{\rho}$.  One finds that $\Lambda_{\rho}$ is a
discrete, free $\FF_q[\theta]$-module of rank $r$ inside
$\CC_{\infty}$.  We call $\Lambda_{\rho}$ the period lattice of
$\rho$, and any element of $\Lambda_{\rho}$ is called a period of
$\rho$.  In 1986, Yu \cite{Yu86} established a fundamental theorem
which asserts that any non-zero period of $\rho$ is transcendental
over $k$ if $\rho$ is defined over $\ok$.

In analogy with de Rham cohomology for elliptic curves, in the late
$1980$'s Anderson, Deligne, Gekeler, and Yu developed a de Rham
cohomology theory for Drinfeld modules (for more details, see
\cite{BP, Gekeler deRham, Thakur, Yu90}). Fix a rank $r$ Drinfeld
$\FF_q[t]$-module $\rho$.  An $\FF_q$-linear map $\delta:
\FF_q[t]\rightarrow \CC_{\infty}[\tau] \tau$ is called a
biderivation if $\delta_{ab}=a(\theta)\delta_{b}+\delta_{a}\rho_{b}$
for any $a,b\in \FF_q[t]$.  A biderivation $\delta$ is uniquely
determined by $\delta_{t}$, and hence the set of all biderivations,
denoted by $D(\rho)$, is a $\CC_{\infty}$-vector space.  Moreover,
for $m\in \CC_{\infty}[\tau]$ the map $\delta^{(m)}:a\mapsto
a(\theta)m- m \rho_{a}$ is also a biderivation.  The biderivation
$\delta^{(m)}$ is called an inner biderivation, and we denote by
$D_{i}(\rho)$ the set of all inner biderivations.  In the case that
$m\in \CC_{\infty}[\tau] \tau$, $\delta^{(m)}$ is called strictly
inner, and the set of all strictly inner biderivations is denoted by
$D_{si}(\rho)$. The de Rham cohomology of $\rho$ is defined to be
the $\CC_\infty$-vector space $H_{DR}(\rho) := D(\rho)/
D_{si}(\rho)$.

Given $\delta \in D(\rho)$, there is a unique entire function
$F_{\delta}$ of the form $F_{\delta}(z) = \sum_{i=1}^{\infty}
c_{i}z^{q^{i}}$ satisfying the functional equation
\begin{equation}\label{E:QuasiPerFunEq}
F_{\delta}(a(\theta)z) - a(\theta) F_{\delta}(z) =
\delta_{a}(\exp_{\rho}(z)),
\end{equation}
for every non-constant $a \in \FF_q[t]$.  We call $F_{\delta}$ the
quasi-periodic function of $\rho$ associated to $\delta$.  It is
$\FF_q$-linear, and furthermore $F_{\delta}|_{\Lambda_\rho} :
\Lambda_\rho \to \CC_\infty$ is $\FF_q[\theta]$-linear.

For the inner biderivation $\delta^{(1)} : a \mapsto
a(\theta)-\rho_{a}$, one has $F_{\delta^{(1)}} = z- \exp_{\rho}(z)$
and hence $F_{\delta^{(1)}}(\lambda)=\lambda$ for $\lambda\in
\Lambda_{\rho}$.  We think of scalar multiples of $\delta^{(1)}$ as
differentials of the first kind.  If $\delta \notin D_i(\rho)$, the
values $F_{\delta}(\lambda)$ for $\lambda\in \Lambda_{\rho}$ are
called quasi-periods of $\rho$.  For any $\delta\in D_{si}(\rho) $
one has $F_{\delta}(\lambda)=0$ for $\lambda\in \Lambda_{\rho}$.
Thus there is a well-defined map of $\CC_\infty$-vector spaces,
\begin{align*}
  H_{DR}(\rho) & \rightarrow \Hom_{\FF_q[\theta]}(\Lambda_{\rho},\CC_{\infty}) \\
  \delta & \mapsto (\lambda\mapsto F_{\delta}(\lambda)).
\end{align*}
which is an isomorphism by Gekeler \cite{Gekeler deRham}. In
particular, $\dim_{\CC_{\infty}} H_{DR}(\rho) = r$.

In $1990$, Yu \cite{Yu90} established a fundamental theorem
concerning the transcendence of quasi-periods. The theorem asserts
that if $\rho$ is defined over $\ok$, then for any $\delta\in
D(\rho)\setminus D_{si}(\rho)$ and $0 \neq \lambda\in
\Lambda_{\rho}$, $F_{\delta}(\lambda)$ is transcendental over $k$.
When we think of $\delta\in D(\rho)\setminus D_{i}(\rho)$ as a
differential of the second kind, Yu's theorem parallels the
classical work of Schneider on elliptic integrals of the second
kind.

The set $\End(\rho) := \{ z\in \CC_{\infty} : z\Lambda_{\rho}
\subseteq \Lambda_{\rho} \}$ is the endomorphism ring of $\rho$, and
if $\End(\rho) \supsetneq \FF_q[\theta]$, then $\rho$ is said to
have complex multiplication. We let $K_\rho$ denote the fraction
field of $\End(\rho)$.

We now fix a rank $2$ Drinfeld $\FF_q[t]$-module $\rho$ defined over
$\ok$ and fix a basis $\{ \omega_{1}, \omega_{2}\}$ of
$\Lambda_{\rho}$ over $\FF_q[\theta]$.  We denote by $F_{\tau}$ the
quasi-periodic function of $\rho$ associated to the biderivation
given by $t\mapsto \tau$. The matrix
\begin{equation}\label{E:DefPrho}
P_{\rho} := \begin{bmatrix}
\omega_{1} & -F_{\tau}(\omega_{1}) \\
\omega_{2} & -F_{\tau}(\omega_{2})
\end{bmatrix}
\end{equation}
is called the period matrix of $\rho$.  As in \cite{BP,Thakur,Yu90},
it is possible to identify the columns of ${P_{\rho}}^{\rm{tr}}$ as
generators of the period lattice of the $2$-dimensional $t$-module
that is the extension of $\rho$ by the additive group $\GG_{a}$ and
which is associated to the biderivation $t \mapsto \tau$.  However,
we do not pursue this further in this paper.  Anderson proved an
analogue of Legendre's relation for periods and quasi-periods of
elliptic curves over $\CC$, namely
\[
 \omega_{1} F_{\tau}(\omega_{2})-\omega_{2}F_{\tau}(\omega_{1})  =
 \frac{\tpi}{\xi},
\]
for some $\xi \in \overline{\FF_q}^{\times}$ satisfying
$\xi^{\frac{1}{q}}=-\xi $ (cf.\ \cite[Thm.~6.4.6]{Thakur}). Here
$\tpi$ is a generator of the period lattice of the Carlitz
$\FF_q[t]$-module $C$ given by $C_{t} = \theta+\tau$.

Now let $\ok(P_{\rho})$ be the field generated by the entries of
$P_{\rho}$.  Thiery \cite{Thiery} has shown that $\trdeg_{\ok}
\ok(P_{\rho})=2$ if $\rho$ has complex multiplication.  For the case
that $\rho$ does not have complex multiplication, based on Yu's
Sub-$t$-module theorem \cite{Yu97}, Brownawell \cite{Brownawell}
proved an analogue of Masser's work on linear independence of
periods and quasi-periods for elliptic curves without complex
multiplication: the $6$ quantities
\[
  1, \tpi, \omega_{1}, \omega_{2}, F_{\tau}(\omega_{1}), F_{\tau}(\omega_{2})
\]
are linearly independent over $\ok$. The first main result of the
present paper is as follows (later stated as
Theorem~\ref{T:PerAlgInd}).

\begin{theorem}\label{Thm1}
Suppose that $p$ is odd.  Let $\rho$ be a rank $2$ Drinfeld
$\FF_q[t]$-module defined over~$\ok$ without complex multiplication.
Then the $4$ quantities
\[
\omega_{1}, \omega_{2}, F_{\tau}(\omega_{1}), F_{\tau}(\omega_{2})
\]
are algebraically independent over $\ok$.
\end{theorem}

The omitted case of $p=2$ provides some added difficulties, which
require further investigation (see Remark~\ref{R:char2}).

\subsection{Algebraic independence of Drinfeld logarithms}
In \cite{Yu97}, Yu proved the full analogue of the qualitative form
of Baker's theorem on logarithms of algebraic numbers for Drinfeld
modules.  Using Yu's Sub-$t$-module Theorem, Brownawell
\cite{Brownawell} proved linear independence of quasi-periods as
well.

\begin{theorem}[Brownawell {\cite[Prop.~2]{Brownawell}};
Yu {\cite[Thm.~4.3]{Yu97}}] \label{T:BrownYu} Let $\rho$ be a
Drinfeld $\FF_q[t]$-module defined over $\ok$. Suppose
$F_{\delta_1}, \dots, F_{\delta_s}$ are quasi-periodic functions for
$\CC_\infty$-linearly independent biderivations $\delta_1, \dots,
\delta_s$ in $D(\rho)/D_i(\rho)$, which are defined over $\ok$.

Let $\lambda_{1}, \dots, \lambda_{m}$ be elements of $\CC_{\infty}$
such that $\exp_{\rho}(\lambda_{i})\in \ok$ for $i=1, \dots, m$. If
$\lambda_{1}, \dots, \lambda_{m}$ are linearly independent over
$K_{\rho}$, then the $1+m(s+1)$ quantities $1$, $\lambda_{1}, \dots,
\lambda_{m}$, $\cup_{j=1}^s \{ F_{\delta_j}(\lambda_1), \dots$,
$F_{\delta_j}(\lambda_m) \}$, are linearly independent over $\ok$.
\end{theorem}

In analogy with a classical conjecture concerning the algebraic
independence of logarithms of algebraic numbers, among the Drinfeld
logarithms of algebraic functions one expects conjecturally that
linear relations over the multiplication ring of the Drinfeld module
in question are the only algebraic relations.

\begin{conjecture}[Brownawell-Yu; see {\cite[p.~323]{Brownawell98}}]
Under the hypotheses of Theorem~\ref{T:BrownYu}, the $m(s+1)$
quantities $\lambda_{1}, \dots, \lambda_{m}$, $\cup_{j=1}^s \{
F_{\delta_j}(\lambda_1), \dots, F_{\delta_j}(\lambda_m) \}$ are
algebraically independent over $\ok$.
\end{conjecture}

The second author has shown that the conjecture holds for the
Carlitz module (since $D(C)/D_i(C)=0$).

\begin{theorem}[Papanikolas {\cite[Thm.~1.2.6]{Papanikolas}}]
Let $C$ be the Carlitz $\FF_q[t]$-module defined by $C_t := \theta +
\tau$. Let $\lambda_{1}, \dots, \lambda_{m}$ be elements of
$\CC_{\infty}$ so that $\exp_C(\lambda_{i}) \in \ok$ for $i=1,
\dots, m$. If $\lambda_{1}, \dots, \lambda_{m}$ are linearly
independent over $k$, then they are algebraically independent over
$\ok$.
\end{theorem}

In the case of odd characteristic, using Theorem~\ref{Thm1} we prove
the above conjecture for logarithms of certain rank~$2$ Drinfeld
modules (later stated as Theorem~\ref{T:LogMain}).

\begin{theorem}\label{Thm2}
Suppose that $p$ is odd.  Let $\rho$ be a rank $2$ Drinfeld
$\FF_q[t]$-module defined over~$\ok$ without complex multiplication.
Let $\lambda_{1}, \dots, \lambda_{m}$ be elements of $\CC_{\infty}$
such that $\exp_{\rho}(\lambda_{i})\in \ok$ for $i=1, \dots, m$. If
$\lambda_{1}, \dots, \lambda_{m}$ are linearly independent over $k$,
then the $2m$ quantities,
\[
\lambda_{1}, \dots, \lambda_{m}, F_{\tau}(\lambda_{1}), \dots,
F_{\tau}(\lambda_{m}),
\]
are algebraically independent over $\ok$.
\end{theorem}

It should be noted that Denis \cite[Thm.~2]{Denis} has shown that at
least $2$ of the $2m$ quantities in the theorem are algebraically
independent over $\ok$.

\subsection{Periods of elliptic curves and elliptic logarithms}
One can compare these results to classical conjectures about elliptic
curves.  Specifically, let $E$ be an elliptic curve over
$\overline{\QQ}$.  Let $\Lambda := \ZZ\omega_{1} + \ZZ\omega_{2}$ be
its period lattice in $\CC$, and let $\wp$ be the Weierstrass
$\wp$-function associated to $\Lambda$. One has the Weierstrass
$\zeta$-function satisfying $\zeta'(z)=-\wp(z)$. Then each $\eta_{i}
:= 2 \zeta(\frac{1}{2}\omega_{i})$, for $i=1$, $2$, is called a
quasi-period of $E$. The matrix
\[
P_{E}:= \begin{bmatrix}
\omega_{1} &\eta_{1} \\
\omega_{2} & \eta_{2}
\end{bmatrix}
\]
is the called period matrix of $E$ and the Legendre relation says
that $\det P_{E} = \pm 2 \pi \sqrt{-1}$.  Conjecturally, one expects
\[
  \trdeg_{\overline{\QQ}} \overline{\QQ}(P_E) =
\begin{cases}
4 & \textnormal{if $\End(E) = \ZZ$,} \\
2 & \textnormal{if $\End(E) \neq \ZZ$.}
\end{cases}
\]
This conjecture is known, by work of Chudnovsky, if $E$ has complex
multiplication, but in the non-CM case, one only knows, by work of
Masser, that the periods and quasi-periods are linearly independent
over $\overline{\QQ}$.

Furthermore, one can conjecture results on elliptic logarithms.  That
is, suppose $\lambda_1, \dots, \lambda_m \in \CC$ satisfy
$\wp(\lambda_i) \in \overline{\QQ}$.  If $\lambda_1, \dots, \lambda_m$
are linearly independent over the multiplication ring of $E$, then one
expects that the $2m$ quantities,
\[
  \lambda_1, \dots, \lambda_m, \zeta(\lambda_1), \dots, \zeta(\lambda_m),
\]
are algebraically independent over $\overline{\QQ}$.  However, the
best known results involve only linear independence over
$\overline{\QQ}$, due to Masser (elliptic logarithms in the CM
case), Bertrand-Masser (elliptic logarithms in the non-CM case), and
W\"ustholz (elliptic integrals of both the first and second kind).
See \cite[\S 4]{Waldschmidt} for more details.

\subsection{Outline of the paper}
Out main tool for proving algebraic independence is rooted in a
linear independence criterion of Anderson, Brownawell, and the
second author \cite{ABP}.  This criterion is key for proving an
algebraic independence theorem of the second author
\cite{Papanikolas}, which relates Galois groups of difference
equations with algebraic relations among their specializations at
$t=\theta$.  The basic structure in this theory is the notion of a
$t$-motive introduced by Anderson \cite{Anderson86}. Therefore, we
review the related background in \S\ref{S:Preliminaries}.  We use
Anderson generating functions to give explicit solutions of the
difference equations associated to a given Drinfeld module,
following a method of Pellarin~\cite{Pellarin}.  By unpublished work
of Anderson, such solutions are necessarily connected to the period
matrix of the Drinfeld module, and we observe this relationship in
this situation.  Hence by the main theorem of \cite{Papanikolas},
Theorem \ref{Thm1} is reduced to calculating the dimension of the
Galois group in question.

The proof of Theorem \ref{Thm1} is in \S\ref{S:PerQPer}. First we
consider the $t$-motive of any rank $2$ Drinfeld module, and then
assuming the characteristic is odd, we establish that the
tautological representation of its Galois group is semisimple.  We also
establish a $t$-motivic version of Tate's conjecture.  With these
properties in hand, we are able to show that if the given Drinfeld
module is without complex multiplication, then its tautological
Galois representation is absolutely irreducible. It follows that the
motivic Galois group in question can be calculated to be $\GL_2$,
and hence Theorem \ref{Thm1} is proved.

Finally, the proof of Theorem \ref{Thm2} occupies \S\ref{S:Logs}. We
construct a suitable $t$-motive such that its period matrix is
related to the logarithms of the algebraic functions in question.
Using the results for Theorem~\ref{Thm1}, we show that the Galois
group in the case of logarithms is an extension of $\GL_2$ by a
vector group.  By calculating its dimension, we prove
Theorem~\ref{Thm2}.

\subsection*{Acknowledgements} We thank G.~Anderson, D.~Thakur, and J.~Yu
for many helpful discussions during the preparation of this
manuscript.  We particularly thank G.~Anderson for sharing his
unpublished notes with us.  We further thank W.-T. Kuo for his
technical advice.  The first author thanks NSC, National Tsing-Hua
University, and NCTS for financial support, and he thanks Texas
A{\&}M University for its hospitality.

\section{Periods and $t$-motives} \label{S:Preliminaries}

\subsection{Notation and Preliminaries} \

\begin{longtable}{p{0.5truein}@{\hspace{5pt}$=$\hspace{5pt}}p{5truein}}
$\FF_q$ &the finite field of $q$ elements, where $q$ is a power of a
prime number $p$. \\
$k$ & $\FF_q(\theta)=$ the rational function field in a variable
$\theta$ over $\FF_q$. \\
$k_{\infty}$ & $\FF_q((\frac{1}{\theta}))$, the completion of  $k$
with respect to the place at infinity. \\
$\overline{k_{\infty}}$ & a fixed algebraic closure of $k_{\infty}$.
\\
$\ok$ & the algebraic closure of $k$ in $\overline{k_{\infty}}$. \\
$\CC_{\infty}$ & the completion of $\overline{k_{\infty}}$ with
respect to the canonical extension of the place at infinity. \\
$|\cdot|$ & a fixed absolute value for the completed field
$\CC_{\infty}$ with $|\theta| = q$. \\
$\TT$ & $\{f\in \CC_{\infty}[[t]] : \textnormal{$f$ converges on
$|t|\leq 1$} \}$ (the Tate algebra of $\CC_\infty$). \\
$\LL$ & the fraction field of $\TT$. \\
$\GG_{a}$ &  the additive group. \\
$\GL_{r/F}$ & for a field $F$, the $F$-group scheme of invertible
$r \times r$ matrices. \\
$\GG_{m}$ & $\GL_{1} =$ the multiplicative group.
\end{longtable}

For $n\in \ZZ$, given a Laurent series $f = \sum_{i}a_{i}t^{i} \in
\CC_{\infty}((t))$ we define the $n$-fold twist of $f$ by the rule
$f^{(n)}:=\sum_{i}a_{i}^{q^{n}}t^{i}$.  For each $n$, the twisting
operation is an automorphism of the Laurent series field
$\CC_{\infty}((t))$ stabilizing several subrings, e.g., $\ok[[t]]$,
$\ok[t]$ and $\TT$. More generally, for any matrix $B$ with entries
in $\CC_{\infty}((t))$ we define $B^{(n)}$ by the rule
${B^{(n)}}_{ij}:=B_{ij}^{(n)}$.

A power series $f = \sum_{i=0}^{\infty} a_{i}t^{i}\in
\CC_{\infty}[[t]]$ that satisfies
\[
\lim_{i \to \infty} \sqrt[i]{|a_{i}|_{\infty} }=0 \quad
\textnormal{and}\quad
[k_{\infty}(a_{0},a_{1},a_{2},\cdots):k_{\infty} ]< \infty
\]
is called an entire power series. As a function of $t$, such a power
series $f$ converges on all of $\CC_{\infty}$ and, when restricted
to $\overline{k_{\infty}}$, $f$ takes values in
$\overline{k_{\infty}}$. The ring of entire power series is denoted
by $\EE$.

\subsection{Tannakian category of $t$-motives}
In this section we follow \cite{Papanikolas} for background and
terminology of $t$-motives.  Let $\ok(t)[\sigma,\sigma^{-1}]$ be the
noncommutative ring of Laurent polynomials in $\sigma$ with
coefficients in $\ok(t)$, subject to the relation
\[
\sigma f=f^{(-1)}\sigma, \quad \forall\,f \in \ok(t).
\]
Let $\ok[t,\sigma]$ be the noncommutative subring of
$\ok(t)[\sigma,\sigma^{-1}]$ generated by $t$ and $\sigma$ over
$\ok$. An Anderson $t$-motive is a left $\ok[t,\sigma]$-module $\cM$
which is free and finitely generated both as a left $\ok[t]$-module
and a left $\ok[\sigma]$-module and which satisfies, for integers
$N$ sufficiently large,
\begin{equation}\label{E:AndersonLieCond}
(t-\theta)^{N}\cM\subseteq \sigma \cM.
\end{equation}
Given an Anderson $t$-motive $\cM$, let $\bm\in \Mat_{r\times
1}(\cM)$ comprise a $\ok[t]$-basis.  Then multiplication by $\sigma$
on $\cM$ is represented by $\sigma \bm=\Phi \bm$ for some matrix
$\Phi\in \Mat_{r}(\ok[t])$. Note that \eqref{E:AndersonLieCond}
implies that $\det\Phi=c(t-\theta)^{s}$ for some $c\in
\ok^{\times}$, $s \geq 0$, and hence $\Phi\in \GL_{r}(\ok(t))$. We
say that $\cM$ is rigid analytically trivial if there exists
$\Psi\in \GL_{r}(\TT)$ so that $\Psi^{(-1)}=\Phi \Psi$.  We note
that $\Psi$ is in fact in $\Mat_{r}(\EE)$ by
\cite[Prop.~3.1.3]{ABP}.

We say that a left $\ok(t)[\sigma,\sigma^{-1}]$-module is a
pre-$t$-motive if it is finite dimensional over $\ok(t)$. Let $\cP$
be the category of pre-$t$-motives. Morphisms in $\cP$ are left
$\ok(t)[\sigma,\sigma^{-1}]$-module homomorphisms. Then $\cP$ forms
an abelian category.

Given a pre-$t$-motive $P$ of dimension $r$ over $\ok(t)$, let
$\bp\in \Mat_{r\times 1}(P)$ be a $\ok(t)$-basis of $P$.
Multiplication by $\sigma$ on $P$ is given by $\sigma\bp=\Phi\bp $
for some matrix $\Phi\in \GL_{r}(\ok(t))$. We say that $P$ is rigid
analytically trivial if there exists $\Psi\in \GL_{r}(\LL)$ so that
$\Psi^{(-1)}=\Phi \Psi$. The matrix $\Psi$ is called a rigid
analytic trivialization for $\Phi$. It is unique up to right
multiplication by a matrix in $\GL_{r}(\FF_q(t))$ \cite[\S
4.1.6]{Papanikolas}.

The Laurent series field $\CC_\infty((t))$ carries the natural
structure of a left $\ok(t)[\sigma,\sigma^{-1}]$-module by setting
$\sigma(f) = f^{(-1)}$.  As such, the subfields $\LL$ and $\ok(t)$
are $\ok(t)[\sigma,\sigma^{-1}]$-submodules.  Similarly
$\CC_\infty[[t]] \supseteq \TT\supseteq \ok[t]$ have natural left
$\ok[t,\sigma]$-module structures.  For any $\sigma$-invariant
subring $F$ of $\CC_\infty((t))$, we denote by $F^{\sigma}$ the
subring of elements in $F$ fixed by $\sigma$. Then we have
\[
\TT^{\sigma} = \ok[t]^{\sigma} = \FF_q[t],\quad \LL^{\sigma} =
\ok(t)^{\sigma} = \FF_q(t).
\]
See \cite[Lem.~3.3.2]{Papanikolas} for more details.

Now consider $P^{\dag}:=\LL\otimes_{\ok(t)}P$, where we give
$P^{\dag}$ a left $\ok(t)[\sigma,\sigma^{-1}]$-module structure by
letting $\sigma$ act diagonally:
\[
\sigma(f\otimes m):=f^{(-1)}\otimes \sigma m, \quad \forall\, f\in
\ok(t), m\in P.
\]
Let
\[
P^{B} := (P^{\dag})^{\sigma} := \{\mu\in P^{\dag} : \sigma \mu=\mu
\}.
\]
Then $P^{B}$ is a vector space over $\FF_q(t)$.  The natural map
$\LL\otimes_{\FF_q(t)}P^{B}\rightarrow P^{\dag}$ is an isomorphism
if and only if $P$ is rigid analytically trivial \cite[\S
3.3]{Papanikolas}. In this situation, $\Psi^{-1}\bp$ is a canonical
$\FF_q(t)$-basis of $P^{B}$, where $\Psi$ is a rigid analytic
trivialization for $\Phi$.

Given an Anderson $t$-motive $\cM$, we obtain a pre-$t$-motive $M$
by setting
\[
M:=\ok(t)\otimes_{\ok[t]}\cM
\]
and extending the action of $\sigma$ in the natural way. Thus,
$\cM\mapsto M$ is a functor from the category of Anderson
$t$-motives to the category of pre-$t$-motives $\cP$.  Moreover, one
has that the natural map
\begin{equation}\label{E:HomMN}
\Hom_{\ok[t,\sigma]}(\cM,\cN)\otimes_{\FF_q[t]} \FF_q(t) \to
\Hom_{\ok(t)[\sigma,\sigma^{-1}]}(M,N)
\end{equation}
is an isomorphism of $\FF_q(t)$-vector spaces for any Anderson
$t$-motives $\cM$ and $\cN$.

We define the category $\cA^{I}$ of Anderson $t$-motives up to
isogeny to be the category whose objects are Anderson $t$-motives
and whose morphisms, for Anderson $t$-motives $\cM$ and $\cN$, are
$\Hom_{\cA^{I}}(\cM,\cN) := \Hom_{\ok[t,\sigma]}(\cM,\cN)
\otimes_{\FF_q[t]} \FF_q(t)$. We define the full subcategory
$\mathcal{AR}^{I}$ of rigid analytically trivial Anderson
$t$-motives up to isogeny by restriction. Letting $\mathcal{R}$ be
the category of rigid analytically trivial pre-$t$-motives,
$\mathcal{R}$ forms a neutral Tannakian category over $\FF_q(t)$
with the fiber functor $P\mapsto P^{B}$.  The functor $\cM\mapsto
M:\mathcal{AR}^{I} \to \mathcal{R}$ is fully faithful.

We define the category $\cT$ of $t$-motives to be the strictly full
Tannakian subcategory of $\mathcal{R}$ generated by the essential
image of the functor $\cM\mapsto M:\mathcal{AR}^{I}\rightarrow
\mathcal{R}$.

For any $t$-motive $M$, let $\cT_{M}$ be the strictly full Tannakian
subcategory of $\cT$ generated by $M$. As $\cT_{M}$ is a neutral
Tannakian category over $\FF_q(t)$, there is an affine algebraic
group scheme $\Gamma_{M}$ over $\FF_q(t)$ so that $\cT_{M}$ is
equivalent to the category of finite dimensional representations of
$\Gamma_{M}$ over $\FF_q(t)$. We call $\Gamma_{M}$ the Galois group
of $M$. The main theorem of \cite{Papanikolas} can be stated as
follows.

\begin{theorem}[Papanikolas
  {\cite[Thm.~1.1.7]{Papanikolas}}] \label{T:TrdegAndDim}
  Let $M$ be a $t$-motive and let $\Gamma_{M}$ be its Galois group.
  Suppose that $\Phi\in \GL_{r}(\ok(t))\cap \Mat_{r}(\ok[t])$
  represents multiplication by $\sigma$ on $M$ and that
  $\det\Phi=c(t-\theta)^{s}$, $c\in \ok^{\times}$. Let $\Psi$ be a
  rigid analytic trivialization of $\Phi$ in $\GL_{r}(\TT)\cap
  \Mat_{r}(\EE)$.  Finally, let $L$ be the subfield of
  $\overline{k_{\infty}}$ generated by all the entries of
  $\Psi(\theta)$ over $\ok$. Then
\[
\dim \Gamma_{M} = \trdeg_{\ok} L.
\]
\end{theorem}

\subsection{Difference Galois groups} \label{S:DiffGal}
Let $M$ be a $t$-motive.  Let $\Phi\in \GL_{r}(\ok(t))$ represent
multiplication by $\sigma$ on $M$ and let $\Psi\in \GL_{r}(\LL)$ be
a rigid analytic trivialization for $\Phi$.  One can develop a
Galois theory for such systems of difference equations,
$\Psi^{(-1)}=\Phi \Psi$, and in turn relate its difference Galois
group to $\Gamma_M$ (see \cite[\S 4]{Papanikolas}).  We describe
this construction below.

We define a $\ok(t)$-algebra homomorphism $\nu_{\Psi} :
\ok(t)[X,1/\det X]\to \LL $ by setting $\nu_{\Psi}(X_{ij}) =
\Psi_{ij}$, where $X=(X_{ij})$ is an $r\times r$ matrix of
independent variables. We let
\[
\mathfrak{p}_{\Psi}:= \Ker \nu_{\Psi}, \quad \Sigma_{\Psi}:=\im
\nu_{\Psi} \subseteq \LL, \quad \Lambda_\Psi := \textnormal{fraction
field of $\Sigma_\Psi$ in $\LL$.}
\]
We set $Z_{\Psi}:= \Spec \Sigma_{\Psi}$.  In this way, $Z_{\Psi}$ is
the smallest closed subscheme of $\GL_{r/\ok(t)}$ so that $\Psi\in
Z_\Psi(\LL)$.

Now set $\Psi_{1}$, $\Psi_{2}\in \GL_{r}(\LL\otimes_{\ok(t)} \LL)$
to be the matrices such that $(\Psi_{1})_{ij}=\Psi_{ij}\otimes 1$
and $(\Psi_{2})_{ij}=1\otimes \Psi_{ij}$, and let
$\tilde{\Psi}:=\Psi_{1}^{-1}\Psi_{2}\in
\GL_{r}(\LL\otimes_{\ok(t)}\LL )$. We define an $\FF_q(t)$-algebra
homomorphism $\mu_{\Psi}: \FF_q(t)[X,1/\det X] \to
\LL\otimes_{\ok(t)}\LL$ by setting $\mu(X_{ij}) =
\tilde{\Psi}_{ij}$.  We let $\Delta_{\Psi}:=\im \mu_{\Psi}$ and set
\begin{equation}\label{E:GammaDef}
\Gamma_{\Psi} := \Spec\Delta_{\Psi}.
\end{equation}
Thus $\Gamma_{\Psi}$ is the smallest closed subscheme of
$\GL_{r/\FF_q(t)}$ such that $\tilde{\Psi} \in
\Gamma_{\Psi}(\LL\otimes_{\ok(t)}\LL)$.

\begin{theorem}[Papanikolas {\cite[\S 4]{Papanikolas}}]
\label{T:GalThy} Let $M$ be a $t$-motive. Let $\Phi\in
{\rm{GL}}_{r}(\ok(t))$ represent multiplication by $\sigma$ on $M$
and let $\Psi\in {\rm{GL}}_{r}(\LL)$ satisfy $\Psi^{(-1)}=\Phi
\Psi$.
\begin{enumerate}
\item[(a)] $\Gamma_{\Psi}$ is a closed $\FF_q(t)$-subgroup scheme of
${\rm{GL}}_{r/\FF_q(t)}$.
\item[(b)] $\Gamma_{\Psi}$ is absolutely irreducible and smooth over
$\oFqt$.
\item[(c)] $Z_{\Psi}$ is stable under right-multiplication by
$\Gamma_{\Psi}\times_{\FF_q(t)}\ok(t)$ and is a torsor for
$\Gamma_{\Psi}\times_{\FF_q(t)}\ok(t)$ over $\ok(t)$.
\item[(d)] $\dim \Gamma_{\Psi}= \trdeg_{\ok(t)} \Lambda_{\Psi}$.
\item[(e)] $\Gamma_{\Psi}$ is isomorphic to $\Gamma_{M}$ over $\FF_q(t)$.
\end{enumerate}
\end{theorem}

\begin{remark}
Theorem \ref{T:GalThy} implies that $\Gamma_{\Psi}$ can be regarded
as a linear algebraic group over $ \FF_q(t)$.
\end{remark}

\subsection{Rank $2$ Drinfeld modules and $t$-motives} \label{S:Rank2}
Let $\rho$ be a rank $2$ Drinfeld $\FF_q[t]$-module defined over
$\ok$ given by $ \rho_{t}=\theta + \kappa \tau + u \tau^{2}$,
$\kappa$, $u \in \ok$. Put
\[
\Phi_{\rho}:=\begin{bmatrix}
  0 & 1 \\
  (t-\theta)/u^{(-2)} & -\kappa^{(-1)}/u^{(-2)}
\end{bmatrix}
\in \Mat_{2}(\ok[t]).
\]
Let $\cM_{\rho}$ be a left $\ok[t]$-module which is free of rank $2$
over $\ok[t]$ with a basis $\bm=[m_{1},m_{2}]^{\mathrm{tr}}\in
\Mat_{2\times 1}(\cM_{\rho})$. We give $\cM_{\rho}$ the structure of
a left $\ok[t,\sigma]$-module by defining $\sigma \bm=\Phi \bm$.

\begin{lemma} \label{L:Mrho}
The left $\ok[t,\sigma]$-module $\cM_{\rho}$ is an Anderson
$t$-motive.
\end{lemma}

\begin{proof}
We claim that $\cM_{\rho}=\ok[\sigma]m_{1}$, whence $\cM_{\rho}$ is
free over $\ok[\sigma]$ since
$\sigma:\cM_{\rho}\rightarrow\cM_{\rho}$ is injective. Note that
since $m_2 = \sigma m_{1}$, we have $m_{2}\in \ok[\sigma]m_{1}$.
Thus, $t m_{1}\in \ok[\sigma]m_{1}$ also. Using that $\sigma t
m_{1}=t  \sigma m_{1}\in \ok[\sigma]m_{1}$, we have that $t m_{2}$
lies in $\ok[\sigma]m_{1}$.  Similarly, $\sigma t m_{2}= t\sigma
m_{2} \in\ok[\sigma]m_{1}$ shows that
$t^{2}m_{1}\in\ok[\sigma]m_{1}$. Continuing this argument we see
that $t^{n}m_{1}$, $t^{n}m_{2}\in \ok[\sigma]m_{1}$ for all $n \geq
0$. This proves our claim. On the other hand, we find easily that
$(t-\theta) \cM_{\rho} \subseteq \sigma \cM_{\rho}$, and hence
$\cM_{\rho}$ is an Anderson $t$-motive.
\end{proof}

Following unpublished work of Anderson, we will show that the
assignment $\rho \mapsto \cM_{\rho}$ is a covariant functor from the
category of Drinfeld modules of rank $2$ over $\ok$ to the category
of Anderson $t$-motives.  First we recall Ore's $\tau$-adjoint
operation \cite[\S 1.7]{Goss},
\[
  f \mapsto f^* : \ok[\tau] \to \ok[\sigma],
\]
where if $f = \sum a_i \tau^i$, then
\[
  f^* := \sum a_i^{(-i)} \sigma^i.
\]
One has that $(fg)^* = g^* f^*$ for all $f$, $g \in \ok[\tau]$, and
moreover $f \mapsto f^*$ defines an anti-isomorphism of rings
$\ok[\tau] \to \ok[\sigma]$.

Suppose $\rho$, $\rho' : \FF_q[t] \to \ok[t]$ are Drinfeld modules
of rank $2$ over $\ok$, and assign $\kappa'$, $u'$, $\{ m_1', m_2' \}$,
$\Phi'$, and $\cM_{\rho'}$ as in the beginning of the section. Now
$e \in \ok[\tau]$ defines a morphism $e: \rho \to \rho'$ if
\[
  e \rho_a = \rho'_a e, \quad \forall\ a \in \FF_q[t].
\]
As we observed in the proof of Lemma~\ref{L:Mrho}, $\cM_\rho =
\ok[\sigma]m_1$ and $\cM_{\rho'} = \ok[\sigma]m_1'$. Define a
$\ok[\sigma]$-linear function
\[
\varepsilon : \cM_\rho \to \cM_{\rho'}
\]
such that $\varepsilon(m_1) = e^* m_1'$.

\begin{lemma}
The map $\varepsilon : \cM_\rho \to \cM_{\rho'}$ is a morphism of
Anderson $t$-motives.
\end{lemma}

\begin{proof}
By definition $\varepsilon$ is $\ok[\sigma]$-linear, so it suffices
to show that it commutes with $t$.  We check easily that
\begin{equation} \label{E:tm1}
  t m_1 = (\rho_t)^* m_1, \quad t m_1' = (\rho'_t)^* m_1'.
\end{equation}
Therefore,
\[
  \varepsilon(t m_1) = (\rho_t)^* \varepsilon(m_1) = (\rho_t)^* e^*
  m_1' = e^* (\rho'_t)^* m_1' = e^* t m_1' = t e^* m_1' = t
  \varepsilon(m_1),
\]
and so $\varepsilon$ is $\ok[t]$-linear.
\end{proof}

It is simple to check that the construction of $\varepsilon$ behaves
well under composition of morphisms, and thus we have defined a
functor $\rho \mapsto \cM_\rho$ as desired.

\begin{proposition} \label{P:DrinAnd}
The functor $\rho \mapsto \cM_\rho$ from rank $2$ Drinfeld modules
over $\ok$ to the category of Anderson $t$-motives is fully
faithful.
\end{proposition}

\begin{proof}
We continue with the notation as above for two rank $2$ Drinfeld
modules $\rho$, $\rho'$.  Certainly $\varepsilon = 0$ if and only if
$e = 0$, and so faithfulness is immediate.  Now suppose $h :
\cM_\rho \to \cM_{\rho'}$ is a morphism.  Then for some $e \in
\ok[\tau]$, $h(m_1) = e^* m_1'$.  Immediately from \eqref{E:tm1} we
see that
\[
  (\rho_t)^* e^* m_1' = h((\rho_t)^* m_1) =
  h(tm_1) = th(m_1) = t e^* m_1' = e^* (\rho'_t)^* m_1'.
\]
Thus $e \rho_t = \rho'_t e$, and $e : \rho \to \rho'$ is a morphism
of Drinfeld modules.  Moreover, $h$ is the morphism $\cM_\rho \to
\cM_{\rho'}$ associated to $e$.
\end{proof}

\begin{remark} \label{R:EndRho}
By Proposition~\ref{P:DrinAnd}, we see for a rank $2$ Drinfeld
module $\rho$ that the ring $\End(\rho)$ is isomorphic to
$\End_{\ok[t,\sigma]}(\cM_{\rho})$. In particular, when $\rho$ does
not have complex multiplication, by \eqref{E:HomMN} we have
$\End_{\ok(t)[\sigma,\sigma^{-1}]}(M_{\rho})=\FF_q(t)$, where
$M_{\rho}:=\ok(t)\otimes_{\ok[t]}\cM_{\rho}$.
\end{remark}

\subsection{Rigid analytic trivializations of rank $2$ Drinfeld
modules}
We now review how to create a rigid analytic trivialization
$\Psi_{\rho}\in \GL_{r}(\TT)\cap \Mat_{2}(\EE)$ for $\Phi_{\rho}$
and connect it with the period matrix of $\rho$ (for more details,
see \cite[\S 4.2]{Pellarin}). For simplicity we assume that $\rho :
\FF_q[t] \to \ok[\tau]$ satisfies
\[
  \rho_t = \theta + \kappa\tau + \tau^2, \quad \kappa \in \ok.
\]
For applications we do not lose any generality (see
Remark~\ref{R:GammaMisGL2}).

Let $\exp_{\rho}(z):=z+\sum_{i=1}^{\infty}\alpha_{i}z^{q^{i}}$ be
the exponential function of $\rho$. Given $u\in \CC_{\infty}$ we
consider the Anderson generating function
\begin{equation}\label{E:AndGF}
f_{u}(t):=\sum_{i=0}^{\infty} \exp_{\rho} \biggl(
\frac{u}{\theta^{i+1}} \biggr)t^{i} =
\sum_{i=0}^{\infty}\frac{\alpha_{i} u^{q^{i}} }{ \theta^{q^{i}}-t }
\in \TT
\end{equation}
and note that $f_{u}(t)$ is a meromorphic function on
$\CC_{\infty}$.  It has simple poles at $\theta$, $\theta^{q},
\ldots$ with residues $-u$, $-\alpha_{1} u^{q}, \ldots$
respectively.  Using that $\rho_{t}(\exp_{\rho}
(\frac{u}{\theta^{i+1}})) = \exp_{\rho}(\frac{u}{\theta^{i}})$, we
have
\begin{equation}\label{E:fu1}
\kappa f_{u}^{(1)}(t)+f_{u}^{(2)} =(t-\theta)f_{u}(t) +
\exp_{\rho}(u).
\end{equation}
Since $f_{u}^{(m)}(t)$ converges away from $\{
\theta^{q^{m}},\theta^{q^{m+1}},\ldots \}$ and
$\Res_{t=\theta}f_{u}(t)=-u$, we have
\begin{equation}\label{E:fu2}
\kappa f_{u}^{(1)}(\theta)+f_{u}^{(2)}(\theta) =-u+\exp_{\rho}(u)
\end{equation}
by specializing \eqref{E:fu1} at $t=\theta$.

Fixing an $\FF_q[\theta]$-basis $\{ \omega_{1},\omega_{2}\}$ of
$\Lambda_{\rho}:=\Ker \exp_{\rho}$, we set
$f_{i}:=f_{\omega_{i}}(t)$ for $i=1$, $2$.  Recall the analogue of
the Legendre relation proved by Anderson,
\begin{equation}\label{E:Legendre}
\omega_{1} F_{\tau}(\omega_{2})- \omega_{2}
F_{\tau}(\omega_{1})=\tpi/ \xi,
\end{equation}
where $\xi\in \overline{\FF_q}^{\times}$ satisfies $\xi^{(-1)}=-\xi$
and $\tpi$ is a generator of the period lattice of the Carlitz
module $C$.  We pick a suitable choice $(-\theta)^{\frac{1}{q-1}}$
of the $(q-1)$-st root of $-\theta$ so that
$\Omega(\theta)=\frac{-1}{\tpi}$, where
\[
\Omega(t):=(-\theta)^{\frac{-q}{q-1}} \prod_{i=1}^{\infty} \biggl(
1-\frac{t}{\theta^{q^{i}}} \biggr) \in \EE,
\]
(see \cite[Cor.~5.1.4]{ABP}).  Now put
\[
\Psi_{\rho}:=\xi \Omega \begin{bmatrix}
  -f_{2}^{(1)} & f_{1}^{(1)} \\
  \kappa f_{2}^{(1)}+f_{2}^{(2)} & -\kappa f_{1}^{(1)}-f_{1}^{(2)}
\end{bmatrix}.
\]
By \eqref{E:fu1} we have $\Psi_{\rho}^{(-1)} =
\Phi_{\rho}\Psi_{\rho}$, and thus $\det(\Psi_\rho)$ is a constant
multiple of $\Omega$.  This implies that $\Psi_\rho \in \GL_2(\TT)$
and that $\cM_\rho$ is rigid analytically trivial.  Hence
$M_{\rho}:= \ok(t) \otimes_{\ok[t]}\cM_{\rho}$ is a $t$-motive.
Since
\[
F_{\tau}(\omega_{i})=\sum_{j=0}^{\infty} {\exp_{\rho} \biggl(
\frac{\omega_{i}}{\theta^{j+1}} \biggr)}^{q}\theta^{j}, \quad i=1,2,
\]
(cf. \cite[\S 6.4]{Thakur}), by evaluating $\Psi_{\rho}$ at
$t=\theta$ we obtain
\begin{equation}\label{E:PsiRhoTheta}
\Psi_{\rho}(\theta)= \frac{\xi}{\tpi} \begin{bmatrix}
F_{\tau}(\omega_{2}) & -F_{\tau}(\omega_{1}) \\
\omega_{2}  & -\omega_{1}
\end{bmatrix}.
\end{equation}
Hence $\Psi_{\rho}^{-1}(\theta)=P_{\rho}$ (cf.\ \eqref{E:DefPrho})
and $\ok(\Psi_{\rho}(\theta)) = \ok(\omega_{1}, \omega_{2},
F_{\tau}(\omega_{1}), F_{\tau}(\omega_{2}))$. Therefore by
Theorem~\ref{T:TrdegAndDim} we have the following equivalence
\begin{equation}\label{E:GL2trdeg4}
\Gamma_{M_{\rho}}=\GL_{2} \Leftrightarrow \trdeg_{\ok} \ok(
\omega_{1}, \omega_{2}, F_{\tau}(\omega_{1}),
F_{\tau}(\omega_{2}))=4.
\end{equation}

\section{Algebraic independence of periods and quasi-periods}
\label{S:PerQPer}

\subsection{The structure of the motivic Galois group $\Gamma_{M}$}
\label{SS:structGM} {From} now on, we fix a rank~$2$ Drinfeld
$\FF_q[t]$-module $\rho$ over $\ok$, given by $\rho_{t}:=\theta+
\kappa \tau+\tau^{2}$, $\kappa\in \ok$.  As in \S\ref{S:Rank2}, we
put
\[
\Phi:=\Phi_{\rho}:=\left[
\begin{matrix}
  0 & 1 \\
 (t-\theta)  & -\kappa^{(-1)} \\
\end{matrix}
\right],
\]
and let $M:=M_{\rho}$ be its associated $t$-motive, with
$\ok(t)$-basis $\bm:=[m_{1},m_{2}]^{\mathrm{tr}} \in \Mat_{2\times
1}(M)$.

\begin{lemma}\label{L:simplicityofM}
Let $\rho$ and $(M,\bm,\Phi)$ be defined as above.  Then $M$ is
simple as a left $\ok(t)[\sigma,\sigma^{-1}]$-module.
\end{lemma}

\begin{proof}
Let $N$ be a non-zero proper left
$\ok(t)[\sigma,\sigma^{-1}]$-submodule of $M$.  Since $N$ is
invariant under the $\sigma$-action, it is spanned over $\ok(t)$ by
$fm_{1}+m_{2}$ for some $f\in \ok(t)^{\times}$.  Because
\[
\sigma(fm_{1}+m_{2}) = \beta (f m_{1}+m_{2})
\]
for some $\beta \in\ok(t)^{\times}$, we have the equalities in
$\ok(t)$,
\[
f^{(-1)} - \kappa^{(-1)}=\beta, \quad (t-\theta)=\beta f.
\]
If $\deg_t(f) > 0$, the first equality implies that $\deg_{t} f =
\deg_{t} \beta$, and therefore by the second equality we have $1 =
2\cdot\deg_{t} f$, which is a contradiction.  If $\deg_t(f) \leq 0$,
then $\deg_t(\beta)\leq 0$, which also contradicts the second
equality.
\end{proof}

We now describe the tautological representation of the Galois group
$\Gamma_M$.  Since $\cT_{M}$ is a neutral Tannakian category over
$\FF_q(t)$, we have a canonically defined faithful representation
\[
\varphi: \Gamma_{M}\hookrightarrow \GL(M^{B}).
\]
Throughout this paper, we always identify $\Gamma_{\Psi}$ with
$\Gamma_{M}$ (cf.\ Theorem \ref{T:GalThy}). The entries of
$\Psi^{-1}\bm$ form a canonical $\FF_q(t)$-basis of $M^B$ (cf.\
\cite[Prop.~3.3.8(b)]{Papanikolas}), and so by
\cite[Thm.~4.5.3]{Papanikolas}, the representation $\varphi$ can be
described as follows: for any $\FF_q(t)$-algebra~$R$,
\begin{equation}\label{E:GaloisRep}
\begin{aligned}
  \varphi: \Gamma_{\Psi}(R) & \hookrightarrow \GL(R \otimes_{\FF_q(t)} M^{B} ) \\
   \gamma  & \mapsto  \bigl((1 \otimes\Psi^{-1}\bm) \mapsto
   (\gamma^{-1} \otimes 1)(1 \otimes \Psi^{-1} \bm)\bigr),
\end{aligned}
\end{equation}
The representation $\varphi$ is called the tautological
representation of $\Gamma_M$.

\begin{proposition}\label{P:SurjDet}
Let $\rho$ and $(M,\bm,\Phi)$ be defined as above.  Then the
determinant map $\det : \Gamma_{M} \to \GG_m$ is surjective.
\end{proposition}

\begin{proof}
We consider the tensor product of $M^{\otimes 2} := M
\otimes_{\ok(t)} M$, on which $\sigma$ acts diagonally. Note that
$\bm \otimes \bm$ defines a canonical $\ok(t)$-basis of $M^{\otimes
2}$ and that the Kronecker product matrix $\Phi \otimes \Phi$
represents multiplication by $\sigma$ on $M^{\otimes 2}$ with
respect to $\bm\otimes \bm$.  Furthermore, $\Psi\otimes\Psi$
provides a rigid analytic trivialization of $\Phi \otimes \Phi$.

Let $N:=\bigwedge^{2}M$ be the space $\ok(t)\cdot\bn$, where $\bn :=
m_1 \otimes m_2 - m_2 \otimes m_1$.  By direct computation, $\sigma
\bn=\det\Phi \cdot \bn$. Hence $N$ is a sub-$t$-motive of
$M^{\otimes 2}$, since
\[
(\det\Psi)^{(-1)}=(\det\Phi)(\det\Psi).
\]
One checks that $\frac{\det\Psi}{\xi \Omega}$ is fixed by $\sigma$,
and hence $\frac{\det\Psi}{\xi \Omega} \in \FF_q(t)^{\times}$. Since
$\Omega$ has infinitely many zeros, $\det\Psi$ is transcendental
over $\ok(t)$.  Thus, by Theorem~\ref{T:GalThy} we see that the
algebraic group $\Gamma_N$ is isomorphic to $\GG_m$. Since $N$ is a
sub-$t$-motive of $M^{\otimes 2}$, we have a surjection
$\Gamma_{M^{\otimes 2}} \twoheadrightarrow \Gamma_N \cong \GG_m$. As
$M^{\otimes 2}$ is an object in the Tannakian category $\cT_{M}$, we
also have a surjective map $\Gamma_{M} \twoheadrightarrow
\Gamma_{M^{\otimes 2}}.$ Composing these two surjective maps, we
have a surjection
\[
\Gamma_{M}\twoheadrightarrow \GG_{m}.
\]
Let $\cT_{M^{\otimes 2}}$ and $\cT_N$ be the Tannakian subcategories
of $\cT_M$ generated by $M^{\otimes 2}$ and $N$ respectively.  Note
that the fiber functor of $\cT_{M^{\otimes 2}}$ is the restriction
of the fiber functor of $\cT_{M}$ to $\cT_{M^{\otimes 2}}$, and the
fiber functor of $\cT_{N}$ is the restriction of the fiber functor
of $\cT_{M^{\otimes 2}}$ to $\cT_{N}$.  Using this property and
\eqref{E:GaloisRep}, we see that, for any $\FF_q(t)$-algebra $R$,
the restriction of the action of any $\gamma \in \Gamma_{M}(R)$ to
the $R$-basis $1\otimes (\det\Psi)^{-1}\bn$ of $R \otimes_{\FF_q(t)}
N^B$ is equal to the action of $\det\gamma$.  Hence, the composition
map $\Gamma_{M} \twoheadrightarrow \GG_m$ is equal to the
determinant map.
\end{proof}

\begin{corollary}\label{C:Solvable}
Let $\rho$ and $(M,\bm,\Phi)$ be as above. If $\dim \Gamma_M \leq
3$, then $\Gamma_{M}$ is solvable.
\end{corollary}

\begin{proof}
By Proposition \ref{P:SurjDet}, we consider the following short
exact sequence of linear algebraic groups
\[
1 \to K \to \Gamma_M \stackrel{\det}{\twoheadrightarrow} \GG_m \to
1.
\]
Let $K^0$ be the identity component of $K$.  Since $K$ is normal in
$\Gamma_M$, for any $g \in \Gamma_M$ we have $g^{-1}K g=K$ and hence
$g^{-1}K^0 g= K^0$.  We note that $K^0$ is solvable since $\dim
K^{0} \leq 2$ (see \cite[p.~137]{Humphreys}), and $\Gamma_M/K^{0}$
is abelian since it is a one-dimensional connected algebraic group
(see \cite[p.~126]{Humphreys}). It follows that $\Gamma_M$ is
solvable.
\end{proof}

\subsection{An analogue of the motivic version of Tate's conjecture}
We let $\End(M)$ denote the endomorphisms of $M$ as a left
$\ok(t)[\sigma,\sigma^{-1}]$-module.  Given $f \in \End(M)$, we have
$f(\bm) = F\bm$ for some $F\in \Mat_2(\ok(t))$.  Since
$f\sigma=\sigma f$, we have
\[
\Phi F=F^{(-1)}\Phi.
\]
{From} this equation we see that the matrix $\Psi^{-1}F\Psi\in
\Mat_2(\LL)$ is fixed by $\sigma$, and hence $\Psi^{-1}F\Psi\in
\Mat_2(\FF_q(t))$. Thus, we have the following injective map:
\begin{align*}
 \End(M)  & \hookrightarrow  \End(M^{B})=\Mat_2(\FF_q(t)), \\
  f & \mapsto  f^{B} := \Psi^{-1}F\Psi.
\end{align*}
Since the representation $\varphi: \Gamma_{\Psi} \to \GL(M^B)$ is
functorial in $M$ (cf.\ \cite[Thm.~4.5.3]{Papanikolas}), it follows
that for any $\gamma\in \Gamma_{\Psi}(\oFqt)$ we have the
commutative diagram
\[
\xymatrix{
 \oFqt\otimes_{\FF_q(t)}M^{B} \ar[r]^{\varphi(\gamma)}
 \ar@{->}[d]^{1\otimes f^{B}}
 & \oFqt\otimes_{\FF_q(t)}M^{B} \ar@{->}[d]^{1\otimes f^{B}}\\
 \oFqt\otimes_{\FF_q(t)}M^{B}
 \ar[r]^{\varphi(\gamma)}& \oFqt\otimes_{\FF_q(t)}M^B.}
\]
Thus, we have defined an injective map
\begin{align*}
 \End(M)  & \hookrightarrow \Cent_{\Mat_2(\FF_q(t))}
 (\Gamma_{\Psi}(\oFqt) ), \\
  f & \mapsto  f^B := \Psi^{-1}F\Psi.
\end{align*}

\begin{theorem} \label{T:TateAnalogue}
The natural map
\[
f \mapsto f^B : \End(M) \to \Cent_{\Mat_2(\FF_q(t))}
(\Gamma_{\Psi}(\oFqt)),
\]
as defined above, is an isomorphism.
\end{theorem}

\begin{proof}
Given $D \in \Cent_{\Mat_2(\FF_q(t))} (\Gamma_{\Psi}(\oFqt))$, we
set $F := \Psi D \Psi^{-1} \in \Mat_2(\Lambda_{\Psi})$.  We claim
that $F$ is in fact in $\Mat_2(\ok(t))$. By
\cite[Thm.~4.4.6(b)]{Papanikolas}, we need only show that the
entries of $F$ are fixed by the action of every $\gamma \in
\Gamma_{\Psi}(\oFqt)$.

For any $\gamma \in \Gamma_{\Psi}(\oFqt)$, the action of $\gamma$ on
$h(\Psi)$, $h \in \overline{k(t)}[X,1/\det X]$, is given by
\[
\gamma * h(\Psi) := h(\Psi \gamma)
\]
(cf.\ \cite[\S4.4.3, \S4.4.4]{Papanikolas}).  Thus the action of
$\gamma$ on entries of $F$ is given as follows:
\[
\gamma * F_{ij} = \bigl((\Psi \gamma)D(\Psi \gamma)^{-1} \bigr)_{ij}
= (\Psi D\Psi^{-1} )_{ij} = F_{ij},
\]
since by definition $D$ commutes with $\gamma$.  Thus $F \in
\Mat_{2}(\ok(t))$, and therefore $F$ induces a $\ok(t)$-linear map
on $M$.

To show that $F \in \End(M)$, it is equivalent to show that
$F^{(-1)}\Phi = \Phi F$. The calculation,
\[
F^{(-1)}\Phi=\Phi \Psi D\Psi^{-1}\Phi^{-1}\Phi=\Phi F,
\]
then completes the proof.
\end{proof}

\subsection{Motivic Galois representations}
In odd characteristic we see in the following theorem that the
tautological representation $\varphi$ is semisimple.

\begin{theorem}\label{T:PhiSemisimplicity}
Suppose that $p$ is odd.  Let $\rho$ be a rank $2$ Drinfeld
$\FF_q[t]$-module over $\ok$ with
\[
\rho_{t}:=\theta+ \kappa \tau+\tau^{2}, \quad \kappa\in \ok.
\]
Let $M$ be the $t$-motive associated to $\rho$ (as in
\S\ref{SS:structGM}). Then the faithful representation
\[
\varphi: \Gamma_{M} \hookrightarrow  \GL \bigl(M^{B}\bigr)
\]
is semisimple.
\end{theorem}

\begin{remark} \label{R:char2}
This theorem is false when $p=2$.  Notably, for the Drinfeld
$\FF_2[t]$-module $\rho$ defined by $\rho_t = \theta +
(\sqrt{\theta} + \theta)\tau + \tau^2$, the above representation of
$\Gamma_{M_\rho}$ is not semisimple.
\end{remark}

To prove Theorem \ref{T:PhiSemisimplicity}, we need the following
reduction Lemma.

\begin{lemma}\label{L:SemisimplicityRed}
Let $\FF$ be the separable closure of $\FF_q(t)$ in $\oFqt$.  Then
the group representation $\varphi_{\FF}: \Gamma_{M}(\FF)
\hookrightarrow \GL ( \FF\otimes_{\FF_q(t)} M^B)$ is semisimple.
\end{lemma}

\begin{proof}
Let $(\bm,\Phi,\Psi)$ be defined as in \S\ref{SS:structGM}.  Suppose
that there exists a one-dimensional $\FF$-vector space $V \subseteq
\FF \otimes_{\FF_q(t)} M^B$, which is invariant under
$\Gamma_{\Psi}(\FF )$.  Let $[u,v]\Psi^{-1}\bm$, for $u$, $v\in \FF
$, be an $\FF$-basis of $V$.

\textbf{Step 1:} We claim that $u\neq 0$ and $v\neq 0$.  On the
contrary, suppose that $u=0$ or $v=0$. Without loss of generality,
we may assume that $u=1$, $v=0$. Let $\mathbb{B}$ be the Borel group
consisting of all lower triangular matrices in $\GL_2$.  Since $V$
is invariant under $\Gamma_{\Psi}(\FF)$, we see that
$\Gamma_{\Psi}(\FF)\subseteq \mathbb{B}(\FF)$. Moreover, since
$\mathbb{B}$ is a closed subgroup of $\GL_{2}$ and
$\Gamma_{\Psi}(\FF)$ is dense in $\Gamma_{\Psi}$ (see
\cite[Lem.~11.2.5]{T.A.Springer}), we see that
\[
\Gamma_{\Psi}\subseteq \mathbb{B}.
\]
Thus $[1,0]\Psi^{-1}\bm\in M^B$ generates a sub-representation of
$\varphi$. This sub-representation corresponds to a nontrivial
proper sub-$t$-motive of $M$ by \cite[Prop.~4.5.8]{Papanikolas},
which contradicts Lemma~\ref{L:simplicityofM}.  Thus we may assume
that $V$ is spanned by
\[
[e,1]\Psi^{-1}\bm, \quad e\in \FF^{\times}.
\]

\textbf{Step 2:} We claim that $e\in \FF^{\times}\setminus
\FF_q(t)^{\times}$.  If $e\in \FF_q(t)^{\times}$, then having $[e,1]
\Psi^{-1} \bm \in M^{B}$ implies that there is a conjugation
embedding over $\FF_q(t)$,
\[
\Gamma_{\Psi}(\FF) \hookrightarrow \mathbb{B}(\FF).
\]
Taking the Zariski closure of $\Gamma_{\Psi}(\FF)$ inside $\GL_2$,
we see that $\Gamma_{\Psi}$ is embedded via conjugation into
$\mathbb{B}$ over $\FF_q(t)$. This provides a sub-representation of
$\varphi$. Using the argument in Step~1, we obtain a contradiction.
Thus $e\in \FF^{\times}\setminus \FF_q(t)^{\times}$.

Now since $e\in  \FF^{\times}\setminus \FF_q(t)^{\times} $, we can
choose an automorphism $\eta$ of the field $\FF$ over $\FF_q(t)$ so
that
\[
\eta(e) \neq  e.
\]

\textbf{Step 3:} We claim that $[\eta(e),1]\Psi^{-1}\bm$ spans an
$\FF$-invariant subspace of $\Gamma_{\Psi}(\FF)$.  Since
$[e,1]\Psi^{-1}\bm$ spans an invariant subspace of
$\Gamma_{\Psi}(\FF)$, for any $\gamma\in \Gamma_{\Psi}(\FF)$ we have
\[
\varphi(\gamma) [e,1] \Psi^{-1}\bm := [e,1] \gamma^{-1}\Psi^{-1}\bm
= \beta_{\gamma}[e,1]\Psi^{-1}\bm,
\]
for some $\beta_{\gamma}\in \FF^{\times}$. Thus we have
\[
[e,1]\gamma^{-1}=\beta_{\gamma}[e,1], \quad \textnormal{for all
$\gamma \in \Gamma_{\Psi}(\FF)$.}
\]
Since $\Gamma_{\Psi}$ is defined over $\FF_q(t)$, the action of
$\eta$ on both sides of the above equation implies that
$[\eta(e),1]\Psi^{-1}\bm$ is a common eigenvector for all $\gamma\in
\Gamma_{\Psi}(\FF)$, which proves the claim.

Since $\eta(e)\neq e$, $[e,1]\Psi^{-1}\bm$ and
$[\eta(e),1]\Psi^{-1}\bm$ are linearly independent over $\FF$, and
hence the group representation $\varphi_{\FF}$ is semisimple.
\end{proof}

\begin{proof}[Proof of Theorem \ref{T:PhiSemisimplicity}] Suppose there exist
$u$, $v\in \oFqt$ so that $[u,v]\Psi^{-1}\bm$ spans a
one-dimensional $\oFqt$-vector space $V$ that is invariant under
$\Gamma_{\Psi}(\oFqt)$. Using the argument in Step~1 of the proof of
the lemma, we see that $uv\neq 0$ and hence $V$ is spanned by
$[e,1]\Psi^{-1}\bm$, $e := u/v$.  Using the argument in Step~2 of
the lemma, we see that
\[
e\in \oFqt^{\times}\setminus \FF_q(t)^{\times}.
\]
We claim that the separable degree of $e$ over $\FF_q(t)$ is
strictly larger than $1$.  Then the proof of Theorem
\ref{T:PhiSemisimplicity} will be completed by using the argument in
Step 3 of the lemma.

Since $[e,1]\Psi^{-1}\bm$ is a common eigenvector for
$\Gamma_{\Psi}(\oFqt)$, we find in particular that for any
\[
\gamma=\left[
\begin{matrix}
  x & y \\
  z & w
\end{matrix}
\right]^{-1}\in \Gamma_{\Psi}(\FF),
\]
we have $[e,1]\gamma^{-1}\Psi^{-1}\bm =\beta_{\gamma}[e,1]\Psi^{-1}
\bm$, for some $\beta_{\gamma}\in\oFqt^{\times}$.  Thus, we have the
equality
\[
[e,1] \left[
\begin{matrix}
  x & y \\
  z & w
\end{matrix}
\right]=\beta_{\gamma}[e,1]
\]
which induces the quadratic relation
\[
ye^{2}+(w-x)e-z=0.
\]
If there exists some $\gamma=\bigl[ \begin{smallmatrix}  x & y \\ z
& w \end{smallmatrix} \bigr]^{-1}\in \Gamma_{\Psi}(\FF)$ with $y\neq
0$, then $e$ must be separable over $\FF_q(t)$ since the
characteristic $p$ is odd and $x$, $y$, $z$, $w\in \FF$.  Hence in
that case the claim above is proved.  Thus we need only consider the
other case that $\Gamma_{\Psi}(\FF)\subseteq \mathbb{B}$. But in
this case, the argument of Step~1 above gives a contradiction, and
hence the proof is completed.
\end{proof}

\subsection{Algebraic independence of periods and quasi-periods}
We continue with the notation of the previous sections, but from now
on we assume that $\rho$ does not have complex multiplication. That
is, we assume $\End(M)= \FF_q(t)$ (cf. Remark \ref{R:EndRho}).  Our
main goal of this section is to prove the following theorem.

\begin{theorem}\label{T:PerAlgInd}
Suppose that $p$ is odd.  Let $\rho$ be a rank $2$ Drinfeld
$\FF_q[t]$-module over $\ok$ with
\[
 \rho_t := \theta + \kappa \tau + \tau^2, \quad \kappa\in \ok,
\]
without complex multiplication.  Let $M$ be the $t$-motive
associated to $\rho$, as in \S\ref{SS:structGM}.  Then
\[
\Gamma_{M} = \GL_{2}.
\]
In particular, the $4$ quantities,
\[
\omega_1, \omega_2, F_{\tau}(\omega_1), F_{\tau}(\omega_2),
\]
are algebraically independent over $\ok$.
\end{theorem}

\begin{remark} \label{R:GammaMisGL2}
The above theorem still holds for an arbitrary rank $2$ Drinfeld
module $\rho$ over $\ok$ without complex multiplication, still under
the assumption $p\neq 2$.  This can be explained as follows.  Let
$u\in \ok^{\times}$ be the coefficient of $\tau^2$ in $\rho_t$. We
pick $x\in \ok^{\times}$ so that $x^{q^2 - 1} = \frac{1}{u}$, and we
define $\nu$ to be the Drinfeld $\FF_q[t]$-module over $\ok$ given
by $\nu_t := x^{-1} \rho_t x$.  That is, $\rho$ is isomorphic to
$\nu$  over $\ok$ (see \cite{Goss,Thakur}). Note that the
coefficient of $\tau^2$ in $\nu_{t}$ is $1$. It is not hard to see
that
\begin{equation}\label{E:expnu}
\exp_{\rho} = x \circ {\exp_{\nu}} \circ x^{-1}, \quad
\Lambda_{\rho} = x \Lambda_{\nu},
\end{equation}
and hence $\nu$ also does not have complex multiplication. We let
$\mathcal{F}_{\tau}$ be the quasi-periodic function of $\nu$
associated to the biderivation defined by $t\mapsto \tau$. Using
\eqref{E:expnu}, one checks that
\begin{equation}\label{E:QPfunnu}
F_{\tau}(z)=x^{q} \mathcal{F}_{\tau}(x^{-1} z).
\end{equation}
More generally, given $\lambda_1, \ldots, \lambda_m \in
\CC_{\infty}$ such that $\exp_{\rho}(\lambda_i)\in \ok$ for $i=1,
\ldots, m$, we  note that by \eqref{E:expnu} we have
$\exp_{\nu}(x^{-1}\lambda_i)\in \ok$ for each $i$. Moreover, from
\eqref{E:QPfunnu} we see that
\[
K:= \ok \bigl(\cup_{i=1}^{m}\{ \lambda_{i}, F_{\tau}(\lambda_{i})\}
\bigr) = \ok \bigl(\cup_{i=1}^{m} \{
x^{-1}\lambda_i,\mathcal{F}_{\tau}(x^{-1} \lambda_i )\} \bigr).
\]
This property will be used in the proof of the main theorem of the
last section.
\end{remark}

The proof of Theorem~\ref{T:PerAlgInd} relies on the following lemma
that shows that, when $\rho$ does not have complex multiplication,
the representation $\varphi$ is absolutely irreducible.

\begin{lemma} \label{L:PhiAbsIrr}
Continuing with the notation of Theorem~\ref{T:PerAlgInd}, the
representation $\varphi: \Gamma_{M}\hookrightarrow \GL (M^{B})$ is
absolutely irreducible.
\end{lemma}

\begin{proof}
Suppose that $\varphi$ is not absolutely irreducible.  Then by
Theorem~\ref{T:PhiSemisimplicity}, we have an embedding by
conjugation over $\oFqt$,
\[
\Gamma_{\Psi}\hookrightarrow \left\{\left[
\begin{matrix}
  * & 0 \\
  0 & *
\end{matrix}
\right]\in \GL_{2}  \right\},
\]
and hence $\Gamma_{\Psi}$ is a (non-split) torus over $\FF_q(t)$.
Since $\Gamma_{\Psi}$ is a torus over $\FF_q(t)$, we see that
\[
\Gamma_{\Psi}(\FF_q(t)) \subseteq \mathbf{C} :=
\Cent_{\Mat_2(\FF_q(t))}\bigl(\Gamma_{\Psi}(\oFqt) \bigr).
\]
Suppose there exists $\gamma\in \Gamma_{\Psi}(\FF_q(t))$ so that
$\gamma$ is not an $\FF_q(t)^{\times}$-scalar multiple of the
identity matrix $\Id_{2}$.  Then $\Id_2$, $\gamma\in \mathbf{C}$ are
linearly independent over $\FF_q(t)$. Hence,
$\dim_{\FF_q(t)}\mathbf{C} \geq 2$, which contradicts
Theorem~\ref{T:TateAnalogue}.

Thus, $\Gamma_{\Psi}(\FF_q(t))$ is contained in the one-dimensional
torus $\GG$ consisting of all $a \Id_{2}$, $a\in \oFqt^{\times} $.
Note that, by \cite[Lem.~13.2.7(ii)]{T.A.Springer},
$\Gamma_{\Psi}(\FF_q(t))$ is dense in $\Gamma_{\Psi}$. Taking the
Zariski closure of $\Gamma_{\Psi}(\FF_q(t))$ inside $\GL_2$, we see
that $\Gamma_{\Psi}\subseteq \GG$ and hence $M^B$ splits. The
splitting of $M^B$ implies that $M$ is not simple by
\cite[Prop.~4.5.8]{Papanikolas}, which contradicts
Lemma~\ref{L:simplicityofM}.
\end{proof}

\begin{proof}[Proof of Theorem \ref{T:PerAlgInd}]
Suppose $\Gamma_{M} \subsetneq \GL_2$.  Then $\dim \Gamma_M \leq 3$,
since $\Gamma_M$ is connected.  By Corollary~\ref{C:Solvable}, we
see that $\Gamma_M$ is solvable, which contradicts the absolute
irreducibility of the representation $\varphi$ from the lemma. Thus,
$\Gamma_{M}= \GL_{2}$.  Moreover, the algebraic independence of
$\omega_{1}$, $\omega_{2}$, $F_{\tau}(\omega_{1})$,
$F_{\tau}(\omega_{2})$ over $\ok$ follows from \eqref{E:GL2trdeg4}.
\end{proof}

\section{Algebraic independence of Drinfeld logarithms}
\label{S:Logs}

\subsection{Some linear algebraic groups} \label{S:GammanDef}
For each $n \geq 0$, we denote by $G_{[n]}$ the $(4+2n)$-dimensional
linear algebraic subgroup of $\GL_{2+n}$ over $\FF_q(t)$:
\[
G_{[n]}:=   \left\{ \left[%
\begin{matrix}
  * & * & 0 &\cdots & 0\\
  * & * & 0&\cdots& 0 \\
  * & * & 1 &\cdots& 0 \\
  \vdots  & \vdots  &  \vdots & \ddots     & \vdots  \\
   * & *  &  0 & \cdots     & 1  \\
\end{matrix}
\right] \in \GL_{2+n} \right\}.
\]
We let
\[
\bX_{n} := \begin{bmatrix}
X_{11} &X_{12} &0 &\cdots  & 0\\
X_{21} &X_{22} &0 & \cdots & 0\\
X_{1} &Y_{1} &1&\cdots  &0 \\
\vdots  & \vdots & \vdots & \ddots & \vdots \\
X_{n} &Y_{n} &0 &\cdots  &1
\end{bmatrix}
\]
be the coordinates of $G_{[n]}$. Throughout this section, we fix a
positive integer $m$ and consider a sequence of linear algebraic
groups $\{\Gamma_{[n]}\}_{0\leq n\leq m }$ over $\FF_q(t)$ with the
following properties:
\begin{itemize}
\item $\Gamma_{[0]}=\GL_{2}$;
\item for each $1\leq n\leq m$, $\Gamma_{[n]}\subseteq G_{[n]}$;
\item $\Gamma_{[n]}$ is absolutely irreducible;
\item we have a surjective morphism $\pi_{n}: \Gamma_{[n]}
\twoheadrightarrow \Gamma_{[n-1]}$, which coincides with the
projection map on the upper left $(n+1) \times (n+1)$ square of
elements in $\Gamma_{[n]}$.
\end{itemize}
In \S\ref{S:AppDrinLogs} we will specify the particular groups
$\Gamma_{[n]}$ that we have in mind, but for now we need only that
they satisfy the above properties.

\begin{definition}
For each $\Gamma_{[n]}$ as above, $\Gamma_{[n]}$ is said to have
\emph{full dimension} if $\dim \Gamma_{[n]}=4+2n$.
\end{definition}

\begin{lemma}\label{L:dimGamman}
Suppose $\{ \Gamma_{[n]}\}_{0\leq n \leq m  }$ is defined as above.
For $1\leq n \leq m$, if $\Gamma_{[n-1]}$ has full dimension, then
\[
\dim\Gamma_{[n]} = \dim\Gamma_{[n-1]} \quad\textnormal{or}\quad
\dim\Gamma_{[n]}= \dim\Gamma_{[n-1]} + 2.
\]
\end{lemma}

\begin{proof}
We consider the short exact sequence of linear algebraic groups
\[
1 \to V \to \Gamma_{[n]} \stackrel{\pi_{n}}{\twoheadrightarrow}
\Gamma_{[n-1]} \to 1.
\]
We note that
\[
V\subseteq   \left\{ \left[
\begin{matrix}
  1 & 0 & 0 &\cdots & 0\\
  0 & 1 & 0&\cdots& 0 \\
  0 & 0 & 1 &\cdots& 0 \\
  \vdots  & \vdots  &  \vdots & \ddots     & \vdots  \\
   * & *  &  0 & \cdots     & 1  \\
\end{matrix}
\right]\right\} \subseteq \GL_{2+n}
\]
and hence $V$ has the natural structure of an additive group.
Moreover, since $\Gamma_{[n-1]}$ has full dimension, for any $a\in
\oFqt^{\times}$ we can pick $\gamma\in \Gamma_{[n]}(\oFqt)$ so that
$\pi_n(\gamma)$ is the block diagonal matrix
\[
\pi_{n}(\gamma)= \begin{bmatrix}
  a & 0 \\
  0 & a
\end{bmatrix} \oplus \Id_{n-1}
\in \Gamma_{[n-1]}(\oFqt).
\]
Direct calculation shows that $\gamma^{-1} v \gamma\in V(\oFqt)$ for
$v\in V(\oFqt)$ and that $V$ is a vector group.

We need only consider the case that $\dim V=1$. We claim that
\[
V\subseteq   \left\{ \left[
\begin{matrix}
  1 & 0 & 0 &\cdots & 0\\
  0 & 1 & 0&\cdots& 0 \\
  0 & 0 & 1 &\cdots& 0 \\
  \vdots  & \vdots  &  \vdots & \ddots     & \vdots  \\
   * & 0  &  0 & \cdots     & 1
\end{matrix}
\right ]\right \} \subseteq \GL_{2+n} \textnormal{\ or\ } V
\subseteq   \left\{ \left[
\begin{matrix}
  1 & 0 & 0 &\cdots & 0\\
  0 & 1 & 0&\cdots& 0 \\
  0 & 0 & 1 &\cdots& 0 \\
  \vdots  & \vdots  &  \vdots & \ddots     & \vdots  \\
  0 & *  &  0 & \cdots     & 1
\end{matrix}
\right ]\right \} \subseteq \GL_{2+n}.
\]
If the claim does not hold, then there exists $v\in V(\oFqt)$ with
both the $(n+2,1)$ and $(n+2,2)$ entries non-zero. Take $a\in
\oFqt$, $a \neq 0$, $1$, and pick any $\delta\in
\Gamma_{[n]}(\oFqt)$ so that $\pi_n(\delta)$ is the block diagonal
matrix
\[
\pi_{n}(\delta)=\begin{bmatrix}
  a & 0 \\
  0 & 1
\end{bmatrix} \oplus \Id_{n-1} \in \Gamma_{[n-1]}(\oFqt).
\]
Then we see that $\delta^{-1}v\delta$ and $v$ are $\oFqt$-linearly
independent vectors in $V(\oFqt)$, which contradicts $\dim V=1$.

Now we pick $\eta\in \Gamma_{[n]}(\oFqt)$ so that $\pi_n$ is the
block diagonal matrix
\[
\pi_{n}(\eta)=\begin{bmatrix}
  0 & 1 \\
  1 & 0
\end{bmatrix} \oplus \Id_{n-1} \in \Gamma_{[n-1]}(\oFqt).
\]
Then the inclusion $\eta^{-1}V(\oFqt) \eta \subseteq V(\oFqt)$ shows
that
\[
V\nsubseteq \left\{ \left[
\begin{matrix}
  1 & 0 & 0 &\cdots & 0\\
  0 & 1 & 0&\cdots& 0 \\
  0 & 0 & 1 &\cdots& 0 \\
  \vdots  & \vdots  &  \vdots & \ddots     & \vdots  \\
   * & 0  &  0 & \cdots     & 1
\end{matrix}
\right ]\right \}  \textnormal{\ and\ } V \nsubseteq   \left\{
\left[
\begin{matrix}
  1 & 0 & 0 &\cdots & 0\\
  0 & 1 & 0&\cdots& 0 \\
  0 & 0 & 1 &\cdots& 0 \\
  \vdots  & \vdots  &  \vdots & \ddots     & \vdots  \\
  0 & *  &  0 & \cdots     & 1
\end{matrix}
\right ]\right \}.
\]
Hence, the dimension of $V$ is either $0$ or $2$. The proof is
completed from the equality $\dim \Gamma_{[n]} =
\dim\Gamma_{[n-1]}+\dim V$.
\end{proof}

Now suppose that $\Gamma_{[n-1]}$ has full dimension for some $1\leq
n \leq m$. We define the following one-dimensional subgroups of
$\Gamma_{[n-1]}\subseteq \GL_{n+1}$:
\begin{gather*}
T_{1}:=\left\{ \left[
\begin{matrix}
  * & 0 & 0 & \cdots & 0 \\
  0 & 1 & 0 & \cdots & 0 \\
  0 & 0 & 1 & \cdots & 0 \\
  \vdots & \vdots & \vdots & \ddots & 0 \\
  0 & 0 & 0 & \cdots & 1
\end{matrix}
\right]\right\},\quad T_{2}:=\left\{ \left[%
\begin{matrix}
  1 & 0 & 0 & \cdots & 0 \\
  0 & * & 0 & \cdots & 0 \\
  0 & 0 & 1 & \cdots & 0 \\
  \vdots & \vdots & \vdots & \ddots & 0 \\
  0 & 0 & 0 & \cdots & 1
\end{matrix}
\right]\right\},
\\
U_{1}:=\left\{ \left[%
\begin{matrix}
  1 & 0 & 0 & \cdots & 0 \\
  * & 1 & 0 & \cdots & 0 \\
  0 & 0 & 1 & \cdots & 0 \\
  \vdots & \vdots & \vdots & \ddots & 0 \\
  0 & 0 & 0 & \cdots & 1
\end{matrix}
\right]\right\},\quad U_{2}:=\left\{ \left[%
\begin{matrix}
  1 & * & 0 & \cdots & 0 \\
  0 & 1 & 0 & \cdots & 0 \\
  0 & 0 & 1 & \cdots & 0 \\
  \vdots & \vdots & \vdots & \ddots & 0 \\
  0 & 0 & 0 & \cdots & 1
\end{matrix}
\right]\right\},
\\
G_{1}:=\left\{ \left[%
\begin{matrix}
  1 & 0 & 0 & \cdots & 0 \\
  0 & 1 & 0 & \cdots & 0 \\
  * & 0 & 1 & \cdots & 0 \\
  \vdots & \vdots & \vdots & \ddots & 0 \\
  0 & 0 & 0 & \cdots & 1
\end{matrix}
\right]\right\},\quad H_{1}:=\left\{ \left[%
\begin{matrix}
  1 & 0 & 0 & \cdots & 0 \\
  0 & 1 & 0 & \cdots & 0 \\
  0 & * & 1 & \cdots & 0 \\
  \vdots & \vdots & \vdots & \ddots & 0 \\
  0 & 0 & 0 & \cdots & 1
\end{matrix}
\right]\right\},
\\
  \vdots
\end{gather*}
\begin{gather*}
G_{n-1}:=\left\{ \left[%
\begin{matrix}
  1 & 0 & 0 & \cdots & 0 \\
  0 & 1 & 0 & \cdots & 0 \\
  0 & 0 & 1 & \cdots & 0 \\
  \vdots & \vdots & \vdots & \ddots & 0 \\
  * & 0 & 0 & \cdots & 1 \\
\end{matrix}
\right]\right\},\quad H_{n-1}:=\left\{ \left[%
\begin{matrix}
  1 & 0 & 0 & \cdots & 0 \\
  0 & 1 & 0 & \cdots & 0 \\
  0 & 0 & 1 & \cdots & 0 \\
  \vdots & \vdots & \vdots & \ddots & 0 \\
  0 & * & 0 & \cdots & 1 \\
\end{matrix}
\right]\right\}.
\end{gather*}
If $n=1$, we define only  $T_{i}$, $U_{i}$ for $i=1$, $2$, without
$G_{j}$, $H_{j}$.

\begin{proposition} \label{P:IsomRatPts}
Suppose $\{ \Gamma_{[n]}\}_{0\leq n \leq m}$ are defined as above,
and suppose that $\Gamma_{[n-1]}$ has full dimension for some $1\leq
n\leq m$. If $\Gamma_{[n]}$ does not have full dimension, then
$\pi_{n}$ induces an isomorphism on $\FF_q(t)$-rational points
$\pi_{n}:\Gamma_{[n]}(\FF_q(t)) \iso \Gamma_{[n-1]}(\FF_q(t))$.
\end{proposition}

\begin{proof}
If we show that the induced tangent map
\[
\mathrm{d}\pi_{n}: \Lie \Gamma_{[n]}\to \Lie \Gamma_{[n-1]}
\]
is a surjective Lie algebra homomorphism, then we have that $\Ker
\pi_{n}$ is defined over $\FF_q(t)$ (see
\cite[Cor.~12.1.3]{T.A.Springer}).  Furthermore, by Lemma
\ref{L:dimGamman}, $\Ker \pi_{n}$ is a zero-dimensional linear space
over $\FF_q(t)$ and hence \cite[Prop.~12.3.4]{T.A.Springer} implies
the isomorphism $\Gamma_{[n]}(\FF_q(t))\cong
\Gamma_{[n-1]}(\FF_q(t))$ induced by $\pi_{n}$.

Let $T_{i}$, $U_{i}$, $G_{j}$, $H_{j}$ be defined as above for
$i=1$, $2$, $j=1, \dots, n-1$, and note that the Lie algebras of
these $2n+2$ algebraic groups span $\Lie \Gamma_{[n-1]}$. To show
the surjection of $\mathrm{d}\pi_{n}$, we need only construct
one-dimensional algebraic subgroups $T_{i}'$, $U_{i}'$, $G_{j}'$,
$H_{j}'$, of $\Gamma_{[n]}$ so that
\[
T_{i}'\cong T_{i}, \quad U_{i}'\cong U_{i}, \quad G_{j}'\cong
G_{j},\quad H_{j}'\cong H_{j}\quad \textnormal{via $\pi_{n}$}
\]
for $i=1$, $2$, $j=1, \dots, n-1$.  Then $\mathrm{d}\pi_{n}:
\Lie\Gamma_{[n]}\twoheadrightarrow \Lie \Gamma_{[n-1]}$ is
surjective since $\Lie(\cdot)$ is a left exact functor of algebraic
groups.

Since $\Ker \pi_{n}$ is a zero-dimensional vector group, $\pi_{n}$
is injective on points. Using this property, one checks directly
that
\begin{itemize}
\item the $Y_{n}$-coordinates of
$\pi_{n}^{-1}(T_{1})$, $\pi_{n}^{-1}(U_{1})$, $\pi_{n}^{-1}(G_{1}),
\dots ,\pi_{n}^{-1}(G_{n-1})$, are zero;
\item the $X_{n}$-coordinates of
$\pi_{n}^{-1}(T_{2})$, $\pi_{n}^{-1}(U_{2})$, $\pi_{n}^{-1}(H_{1}),
\dots, \pi_{n}^{-1}(H_{n-1})$ are zero.
\end{itemize}
To construct $T_{1}'$, $T_{2}'$, we let $a$, $b\in \oFqt^{\times}
\setminus \overline{\FF_q}^{\times }$ and pick $\gamma_{1}$,
$\gamma_{2}\in \Gamma_{[n]}(\oFqt)$ so that
\[
\pi_{n}(\gamma_{1})=\begin{bmatrix}
a &0 & \cdots & 0\\
0 &1 & \cdots &0 \\
\vdots  & \vdots & \ddots & \vdots \\
0  &0 &\cdots  & 1
\end{bmatrix}
\textnormal{\ and\ }
\pi_{n}(\gamma_{2})= \begin{bmatrix}
1 &0 & \cdots & 0\\
0 &b & \cdots &0 \\
\vdots  & \vdots & \ddots & \vdots \\
0  &0 &\cdots  & 1
\end{bmatrix}.
\]
We let $T_{i}'$ be the Zariski closure of the cyclic subgroup of
$\Gamma_{[n]}$ generated by $\gamma_{i}$ (inside $\Gamma_{[n]}$) for
$i=1$, $2$.  Then one checks directly that $T_{i}'$ is a
one-dimensional torus in $\Gamma_{[n]}$, and the restriction of
$\pi_{n}$ to $T_{i}'$ induces an isomorphism $T_{i}'\cong T_{i}$ for
$i=1,2$, (cf.\ \cite[\S 6.2.4]{Papanikolas}). More precisely, the
defining equations of $T_{1}'$, $T_{2}'$ can be written as follows:
\begin{equation}\label{E:T1primeDefEqs}
T_{1}':\left \{
\begin{array}{ll}
  (a-1)X_{n}-c(X_{11}-1)=0,& X_{12}=0 \\
  X_{21}=0 ,& X_{22}-1=0\\
  X_{1}=0,& Y_{1}=0 \\
  \quad\vdots& \quad\vdots \\
  X_{n-1}=0,&Y_{n}=0
\end{array}
\right\},
\end{equation}
\begin{equation} \label{E:T2primeDefEqs}
 T_{2}':\left \{
\begin{array}{ll}
  X_{11}-1=0, & X_{12}=0\\
  X_{21}=0 ,& (b-1)Y_{n}-d(X_{22}-1)=0\\
  X_{1}=0 ,&Y_{1}=0 \\
  \quad\vdots & \quad\vdots\\
  X_{n}=0,& Y_{n-1}=0
\end{array}
\right\},
\end{equation}
where $c$ is the $(n+2,1)$-entry of $\gamma_{1}$ and $d$ is the
$(n+2,2)$-entry of $\gamma_{2}$.

For the constructions of $U_{1}'$, $U_{2}'$, we let $u_{i}\in
U(\FF_q(t))$ be an $\FF_q(t)$-rational basis for the one-dimensional
vector group $U_{i}$  and pick $u_{i}'\in \Gamma_{[n]}(\oFqt)$ so
that $\pi_{n}(u_{i}')=u_{i}$ for $i=1$, $2$. We define $U_{i}'$ to
be the one-dimensional vector group in $\Gamma_{[n]}$ via the
conjugations
\[
\eta_{1}^{-1}u_{1}\eta_{1}, \quad \eta_{2}^{-1}u_{2}\eta_{2}, \quad
\textnormal{for\ } \eta_{i}\in T_{i}',\ i=1,2.
\]
Then we see that $U_{i}'\cong  U_{i}$ via $\pi_{n}$ for $i=1$, $2$.

Finally we use the method above as well as conjugations to construct
the desired $G_{j}'$, $H_{j}'$ so that $G_{j}'\cong  G_{j}$ and
$H_{j}'\cong H_{j}$ via $\pi_{n}$, for $j=1,\dots ,n-1$.  The
arguments are essentially the same as the constructions of $T_i'$
and $U_i'$, and we omit the details.
\end{proof}

\subsection{Defining equations for $\Gamma_{[n]}$}
\label{S:GammanDefEqs} We continue with the notation of the previous
section, and we assume that $\Gamma_{[n-1]}$ has full dimension for
some $1 \leq n \leq m$.  Furthermore, we assume that $\Gamma_{[n]}$
does not have full dimension, and so by Lemma~\ref{L:dimGamman},
$\dim \Gamma_{[n]} = \dim \Gamma_{[n-1]}$.  Since we have shown that
$\pi_n :\Gamma_{[n]}(\FF_q(t))\iso \Gamma_{[n-1]}(\FF_q(t))$, for
any $a$, $b \in \FF_q(t)^{\times}\setminus \FF_q^{\times}$ we can
pick $\gamma_{1}$, $\gamma_{2}\in \Gamma_{[n]}(\FF_q(t))$ so that
\[
 \pi_{n}(\gamma_{1})=\begin{bmatrix}
a &0 & \cdots & 0\\
0 &1 & \cdots &0 \\
\vdots  & \vdots & \ddots & \vdots \\
0  &0 &\cdots  & 1
\end{bmatrix}
\textnormal{\ and\ }
\pi_{n}(\gamma_{2})= \begin{bmatrix}
1 &0 & \cdots & 0\\
0 &b & \cdots &0 \\
\vdots  & \vdots & \ddots & \vdots \\
0  &0 &\cdots  & 1
\end{bmatrix}.
\]
If we let $T_{i}'$ be the Zariski closure of the cyclic group
generated by $\gamma_{i}$ (inside $\Gamma_{[n]}$), then the defining
equations of $T_{i}'$ are given as in \eqref{E:T1primeDefEqs} and
\eqref{E:T2primeDefEqs} for $i=1$, $2$. We shall note that $c$ and
$d$ in \eqref{E:T1primeDefEqs} and \eqref{E:T2primeDefEqs} are in
$\FF_q(t)$.

Let $V_{1}$ be the $n$-dimensional vector group over $\FF_q(t)$ in
$\Gamma_{[n-1]}$ spanned by $U_{1}$, $G_{1}, \dots, G_{n-1}$ given
as above. Let $U_{1}'$, $G_{1}', \dots, G_{n-1}'$ be the
one-dimensional vector groups in $\Gamma_{[n]}$ given as in the
proof of Proposition \ref{P:IsomRatPts}.  Since we have shown that
$\pi_n:\Gamma_{[n]}(\FF_q(t))\iso \Gamma_{[n-1]}(\FF_q(t))$, we see
that these one-dimensional vector groups are defined over
$\FF_q(t)$, and since $\mathrm{d}\pi_n$ is surjective they are
defined by linear equations. Let $V_{1}'$ be the $n$-dimensional
vector group over $\FF_q(t)$ in $\Gamma_{[n]}$ spanned by $U_{1}'$,
$G_{1}', \dots, G_{n-1}'$.

Similar to the constructions above, we let $V_{2}$ be the
$n$-dimensional vector group over $\FF_q(t)$ in $\Gamma_{[n-1]}$
spanned by $U_{2}$, $H_{1}, \dots, H_{n-1}$.  Let
$U_{2}'$, $H_{1}',\dots, H_{n-1}'$ be the one-dimensional vector groups
in $\Gamma_{[n]}$ given as in the proof of Proposition
\ref{P:IsomRatPts}.  We define $V_{2}'$ to be the $n$-dimensional
vector group over $\FF_q(t)$ in $\Gamma_{[n]}$ spanned by
$U_{2}'$, $H_{1}', \dots, H_{n-1}'$ and note that $V_{i}'$ is isomorphic
to $V_{i}$ via $\pi_{n}$ for $i=1$, $2$.

Note that the defining equations of $V_{1}'$, $V_{2}'$ are given as
follows:
\begin{gather*}
V_{1}':\left\{
\begin{array}{l}
  X_{11}-1=0,\ r_{21}X_{21}+r_{1}X_{1}+\cdots+r_{n}X_{n}=0 \\
  X_{12}=0,\ X_{22}-1=0,\ Y_{1}=0,\dots,Y_{n}=0
\end{array}\right\},
\\
V_{2}':\left\{
\begin{array}{l}
  X_{11}-1=0,\ X_{21}=0,\ X_{1}=0, \dots,X_{n}=0 \\
  X_{22}-1=0,\ s_{12}X_{12}+s_{1}Y_{1}+\dots+s_{n}Y_{n}=0
\end{array}\right\},
\end{gather*}
for some $r_{21}$, $r_{1}, \dots, r_{n}$, $s_{12}$, $s_{1}, \dots,
s_{n} \in \FF_q(t)$.  Note that, since $V_{i}'$ is isomorphic to
$V_{i}$ via $\pi_{n}$ for $i=1$, $2$, we have that $r_{n}\neq 0$,
$s_{n}\neq 0$.

We define $P_{i}'$ to be the Zariski closure of the subgroup
generated by $T_{i}'$ and $V_{i}'$ inside $\Gamma_{[n]}$ for $i=1$,
$2$.  Then we see that for each $i=1$, $2$, $P_{i}'$ is the
$(n+1)$-dimensional affine linear space containing $T_{i}'$ and
$V_{i}'$ (cf.\ \cite[\S 6.2.4]{Papanikolas}), and hence their
defining equations can be described as follows:
\begin{equation} \label{E:P1primeP2prime}
\begin{gathered}
P_{1}':\left\{
\begin{array}{l}
\phi_{1}:=(a-1) (r_{21}X_{21}+r_{1}X_{1} + \cdots +
r_{n}X_{n})-c r_{n}(X_{11}-1)=0, \\
X_{12}=0,\ X_{22}-1=0,\ Y_{1}=0,\dots,Y_{n}=0
\end{array}\right\}
\\
P_{2}':\left\{
\begin{array}{l}
  X_{11}-1=0,\ X_{21}=0,X_{1}=0,\dots,X_{n}=0, \\
  \phi_{2}:=(b-1) (s_{12}X_{12}+s_{1}Y_{1}+ \cdots +
s_{n}Y_{n})-d s_{n}(X_{22}-1)=0
\end{array}\right\}.
\end{gathered}
\end{equation}
Note that $\dim P_{i}'=n+1$ for $i=1$, $2$. Consider now the morphism
defined by the product of matrices,
\[
P_{1}'\times  P_{2}' \to \Gamma_{[n]}.
\]
Its image is denoted by $P_{1}' \cdot P_{2}'$. We claim that
the Zariski closure $\overline{P_{1}'\cdot P_{2}'}$ of $P_{1}'\cdot
P_{2}'$ inside $\Gamma_{[n]}$ is all of $\Gamma_{[n]}$.

To prove this claim, we define $P_{i}$ to be the Zariski closure of
the subgroup of $\Gamma_{[n-1]}$ generated by $T_{i}$ and $V_{i}$
for $i=1$, $2$.  That is, since $\Gamma_{[n-1]}$ has full dimension,
\[
P_{1}:=\left\{ \left[%
\begin{matrix}
  * & 0 & 0 & \cdots & 0 \\
  * & 1 & 0 & \cdots & 0 \\
  * & 0 & 1 & \cdots & 0 \\
  \vdots & \vdots & \vdots & \ddots & 0 \\
  * & 0 & 0 & \cdots & 1
\end{matrix}
\right]\right\}\subseteq G_{[n-1]},\quad
P_{2}:=\left\{ \left[%
\begin{matrix}
  1 & * & 0 & \cdots & 0 \\
  0 & * & 0 & \cdots & 0 \\
  0 & * & 1 & \cdots & 0 \\
  \vdots & \vdots & \vdots & \ddots & 0 \\
  0 & * & 0 & \cdots & 1
\end{matrix}
\right]\right\}\subseteq G_{[n-1]}.
\]
Direct calculation shows that
\[
G_{[n-1]}=\overline{P_{1}\cdot P_{2}}\cup V(X_{11}),
\]
where $V(X_{11})$ is the closed subvariety of $G_{[n-1]}$ given by
$X_{11}=0$. Hence, $\Gamma_{[n-1]}=\overline{P_{1}\cdot P_{2}}$ since
$\Gamma_{[n-1]} \subseteq G_{[n-1]}$ is assumed to be irreducible of
maximal dimension.

Furthermore, since we have a surjective map
$P_{i}'\twoheadrightarrow P_{i}$ induced by $\pi_{n}$ for $i=1$,
$2$, the restriction of $\pi_{n}$ to $\overline{P_{1}' \cdot
P_{2}'}$ is dominant and hence $\dim \overline{P_{1}' \cdot P_{2}' }
\geq 2+2n$.  From the assumptions that $\dim \Gamma_{[n]} = \dim
\Gamma_{[n-1]}=2+2n$ and that $\Gamma_{[n]}$ is irreducible, we see
that $\overline{P_{1}' \cdot
  P_{2}'}=\Gamma_{[n]}$.

On the other hand, one can similarly show that
$\overline{P_{2}'\cdot P_{1}'}=\Gamma_{[n]}$.  We omit the details.

Now, we claim that $\phi_{1}$, $\phi_{2}$, as defined in
\eqref{E:P1primeP2prime}, give rise to defining equations for
$\Gamma_{[n]}$. For polynomials $g_{1},\dots,g_{m}\in
\FF_q(t)[\bX_{n}]$, we let $ V(g_{1},\dots,g_{m})$ be the closed
subvariety of $G_{[n]}$ given by $g_{1}=\cdots=g_{m}=0$.  Given any
\begin{equation} \label{E:p1p2}
\bp_{1}=\begin{bmatrix}
  x_{11} & 0 & 0 & \cdots & 0 \\
  x_{21} & 1 & 0 & \cdots & 0 \\
  x_{1} & 0 & 1 & \cdots & 0 \\
  \vdots & \vdots & \vdots & \ddots & 0 \\
  x_{n} & 0 & 0 & \cdots & 1
\end{bmatrix}
\in P_{1}'
\textnormal{\ and\ }
\bp_{2}=\begin{bmatrix}
  1 & x_{12} & 0 & \cdots & 0 \\
  0 & x_{22} & 0 & \cdots & 0 \\
  0 & y_{1} & 1 & \cdots & 0 \\
  \vdots & \vdots & \vdots & \ddots & 0 \\
  0 & y_{2} & 0 & \cdots & 1
\end{bmatrix}
\in P_{2}',
\end{equation}
from the matrix product,
\[
\bp_{1} \bp_{2}= \begin{bmatrix}
  x_{11} & * & 0 & \cdots & 0 \\
  x_{21} & * & 0 & \cdots & 0 \\
  x_{1} & * & 1 & \cdots & 0 \\
  \vdots & \vdots & \vdots & \ddots & 0 \\
  x_{n} & * & 0 & \cdots & 1
\end{bmatrix}
\in P_{1}'\cdot P_{2}',
\]
we see that $V(\phi_{1})\supseteq P_{1}'\cdot  P_{2}'$,
and hence $ V(\phi_{1})\supseteq \Gamma_{[n]}$.
A similar calculation using the matrix product $\bp_2 \bp_1 \in P_2'
\cdot P_1'$ shows that $V(\phi_2) \supseteq P_2'\cdot P_1'$.
Furthermore, because $\phi_{1}$, $\phi_{2}$ are degree one
polynomials, $ V(\phi_{1},\phi_{2}) $ is irreducible, and therefore
$V(\phi_{1},\phi_{2})= \Gamma_{[n]}$, as claimed.

Let $\phi_{1} := \ell_{11}X_{11} + \ell_{21}X_{21} + \ell_{1}X_{1} +
\dots + \ell_{n}X_{n} - \ell_{11}$, where
\[
\ell_{11}=-cr_{n},\quad \ell_{21}=(a-1)r_{21},\ell_{i}=(a-1)r_{i},
\quad i=1,\dots,n.
\]
Note that all coefficients of $\phi_{1}$ are in $\FF_q(t)$
and $\ell_{n}\neq 0$.  Finally we claim that, without loss of
generality, $\phi_{2}$ can be written as
\[
\phi_{2} = \ell_{11}X_{12} + \ell_{21}X_{22} + \ell_{1}Y_{1} + \cdots
+ \ell_{n} Y_{n} - \ell_{21}.
\]
To verify the claim, we note that, modulo $(b-1)$, without loss of
generality we may write $\phi_{2}$ as $\phi_{2} = s_{12}X_{12} +
s_{22}X_{22} + \sum_{i=1}^{n} s_{i}Y_{i} - s_{22}$, where
$s_{22}:=\frac{-d s_{n}}{b-1}$.  Choose an element of the form
\[
\bp_{2}= \begin{bmatrix}
  1 & 1 & 0 & \cdots & 0 \\
  0 & 1 & 0 & \cdots & 0 \\
  0 & 0 & 1 & \cdots & 0 \\
  \vdots & \vdots & \vdots & \ddots & 0 \\
  0 & y_{n} & 0 & \cdots & 1
\end{bmatrix}
\in P_{2}'(\oFqt), \quad \textnormal{i.e., $\phi_{2}({\bf p}_{2})=0.$}
\]
We choose an arbitrary $\bp_1 \in P_1'(\oFqt)$ as in \eqref{E:p1p2},
and because $\phi_1(\bp_1) = 0$ and $\bp_{1}\bp_{2} \in
V(\phi_{1},\phi_{2})$, a direct calculation shows that
\[
s_{12}x_{11} + s_{22}x_{21} + \sum_{i=1}^{n}s_{i}x_{i} -s_{12}=0.
\]
Namely, $P_{1}'$ is contained in the $(n+1)$-dimensional affine
linear space given by
\[
V\biggl( s_{12}X_{11} + s_{22}X_{21} + \sum_{i=1}^{n} s_{i}X_{i}-s_{12},\
X_{12},\ X_{22}-1,\ Y_{1},\dots,Y_{n} \biggr).
\]
Since $P_{1}'$ is also an affine linear space of dimension $n+1$, we
see that the two vectors
\[
(\ell_{11},\ell_{21},\ell_{1},\dots,\ell_{n}),\quad
(s_{12},s_{22},s_{1},\cdots ,s_{n} )
\]
are parallel, which completes the claim.  We summarize the
investigations of this section in the following lemma.

\begin{lemma}\label{L:GammanEqs}
  Let $\{ \Gamma_{[n]}\}_{0\leq n \leq m}$ be a sequence of groups
  defined as in \S\ref{S:GammanDef}. Suppose that $\Gamma_{[n-1]}$ has
  full dimension for some $1\leq n \leq m$ but that $\dim
  \Gamma_{[n]}= \dim\Gamma_{[n-1]}$.  Then there exist
  $\ell_{11},\ell_{21},\ell_{1},\cdots,\ell_{n}\in \FF_q(t)$ with
  $\ell_{n}\neq 0$ so that
\begin{gather*}
  \phi_{1} := \ell_{11}X_{11} + \ell_{21}X_{21} + \ell_{1}X_{1}
+ \dots + \ell_{n}X_{n} - \ell_{11},   \\
  \phi_{2} := \ell_{11}X_{12} + \ell_{21}X_{22} + \ell_{1}Y_{1}
+ \dots + \ell_{n}Y_{n}-\ell_{21}
\end{gather*}
are defining polynomials for $\Gamma_{[n]}$.
\end{lemma}

\subsection{Application to Drinfeld logarithms}
\label{S:AppDrinLogs}

We fix a rank $2$ Drinfeld $\FF_q[t]$-module $\rho$ over $\ok$ with
$\rho_{t}:=\theta + \kappa \tau+\tau^{2}$, $\kappa\in \ok$, and we fix
a $\FF_q[\theta]$-basis $\{ \omega_{1},\omega_{2} \}$ of the period
lattice $\Lambda_{\rho} := \Ker \exp_{\rho}$. We let $\Phi :=
\Phi_{\rho}$, $\Psi := \Psi_{\rho}$, $\xi$, and $\Omega$ be defined as
in \S\ref{S:Rank2}.

Given $\lambda\in \CC_{\infty}$ with $\exp_{\rho}(\lambda)
=:\alpha\in \ok$, let $f_{\lambda}$ be the Anderson generating
function associated to $\lambda$ as in \eqref{E:AndGF}. We
define
\[
\bg := \begin{bmatrix}
g_{1} \\
g_{2}
\end{bmatrix}
:=
\begin{bmatrix}
-\kappa f_{\lambda}^{(1)}-f_{\lambda}^{(2)}  \\
-f_{\lambda}^{(1)}
\end{bmatrix}
\in \Mat_{2\times 1}(\TT).
\]
By \eqref{E:fu1} and \eqref{E:fu2} we see that
\begin{equation}\label{E:LogFnEq}
\Phi^{\mathrm{tr}}\bg^{(-1)}=
  \bg +
\begin{bmatrix}
\alpha \\ 0
\end{bmatrix},
\quad
g_{1}(\theta) = \lambda-\alpha,
\quad
g_{2}(\theta)=-F_{\tau}(\lambda).
\end{equation}
Given $\lambda_{1},\dots,\lambda_{m}\in \CC_{\infty}$ with
$\exp_{\rho}(\lambda_{i}) =: \alpha_{i}\in \ok$ for $i=1, \dots, m$,
for each $1\leq n \leq m$ let
\[
\bg_{n}:=
\begin{bmatrix}
g_{n1} \\
g_{n2}
\end{bmatrix}
:= \begin{bmatrix}
-\kappa f_{\lambda_{n}}^{(1)}-f_{\lambda_{n}}^{(2)}  \\
-f_{\lambda_{n}}^{(1)}
\end{bmatrix}
\quad\textnormal{and}\quad
\mathbf{h}_{n}:=
\begin{bmatrix}
\alpha_n \\ 0
\end{bmatrix}.
\]
We further define
\[
\Phi_{n}:=\begin{bmatrix}
\Phi & \mathbf{0} & \cdots & \mathbf{0} \\
\bh_{1}^{\mathrm{tr}} & 1 & \cdots & 0 \\
\vdots &\vdots  & \ddots & \vdots \\
\bh_{n}^{\mathrm{tr}} & 0 &\cdots  & 1 \\
\end{bmatrix}
\in\Mat_{2+n}(\ok[t]),\quad
\Psi_{n}:=\begin{bmatrix}
\Psi & \mathbf{0} & \cdots & \mathbf{0} \\
\bg_{1}^{\mathrm{tr}}\Psi & 1 & \cdots & 0 \\
\vdots & \vdots & \ddots & \vdots \\
\bg_{n}^{\mathrm{tr}}\Psi & 0 & \cdots & 1 \\
\end{bmatrix}
\in\GL_{2+n}(\TT).
\]
Using \eqref{E:LogFnEq} we have $\Psi_{n}^{(-1)}=\Phi_{n} \Psi_{n}$
and thus
\begin{equation}\label{E:okPsintheta}
\ok(\Psi_{n}(\theta)) = \ok(\omega_{1}, \omega_{2}, \lambda_{1}, \dots,
\lambda_{n}, F_{\tau}(\omega_{1}), F_{\tau}(\omega_{2}),
F_{\tau}(\lambda_{1}),\dots,F_{\tau}(\lambda_{n})).
\end{equation}

\begin{proposition}
  For each $1\leq n \leq m$, let $\Phi_{n}$ be defined as above. Then
  $\Phi_{n}$ defines a $t$-motive $M_{n}$.
\end{proposition}

\begin{proof}
  By Lemma \ref{L:Mrho}, $\Phi$ itself defines an
  Anderson $t$-motive $\cM_{\rho}$.  Using this the proof is then
  essentially the same as the proof of
  \cite[Prop.~6.1.3]{Papanikolas}.  We omit the details.
\end{proof}

Suppose that $\rho$ does not have complex multiplication. Let
$M_{0}:=M_{\rho}$ be the $t$-motive associated to $\rho$ (see
\S\ref{S:Rank2}), and let $\Gamma_{[0]} := \Gamma_{M_{0}} =\GL_{2}$ be
its Galois group (see Theorem \ref{T:PerAlgInd}).  For each $1\leq
n\leq m$, let $\Gamma_{[n]}$ be the Galois group of $M_{n}$.  By
\eqref{E:GammaDef}, we see that $\Gamma_{[n]}\subseteq G_{[n]}$.
Moreover, since $M_{n-1}$ is a sub-$t$-motive of $M_{n}$, we have a
surjective map
\[
\pi_{n}:\Gamma_{[n]}\twoheadrightarrow \Gamma_{[n-1]}
\]
(cf.\ proof of Proposition~\ref{P:SurjDet}).  More precisely, for
any $\FF_q(t)$-algebra $R$ the restriction of the action of any
$\gamma\in \Gamma_{[n]}(R)$ to $R\otimes_{\FF_q(t)} M_{n-1}^{B}$ is
the same as the action of the upper left $(n+1) \times (n+1)$ square
of $\gamma$.  Thus we see that the map $\pi_{n}$ coincides with the
projection map on the upper left $(n+1) \times (n+1)$ square of
elements in $\Gamma_{[n]}$. Thus, the sequence $\{ \Gamma_{[n]}
\}_{0\leq n\leq m}$ satisfies the defining conditions of
\S\ref{S:GammanDef}, and the results of \S\ref{S:GammanDef} and
\S\ref{S:GammanDefEqs} apply.

Finally, by Theorem \ref{T:TrdegAndDim} we note that $\dim
\Gamma_{[n]}=\trdeg_{\ok} \ok(\Psi_{n}(\theta))$ for each $0 \leq n
\leq m$.  Combining Theorem~\ref{T:PerAlgInd}, Lemma
\ref{L:dimGamman}, and \eqref{E:okPsintheta}, we now prove a lemma
that is the heart of Theorem~\ref{T:LogMain} which follows.

\begin{lemma}\label{L:redII}
  Suppose that $p$ is odd.  Let $\rho$ be a Drinfeld $\FF_q[t]$-module
  without complex multiplication with $\rho_{t}=\theta+ \kappa \tau
  +\tau^{2}$, $\kappa \in \ok$.  Let $\Lambda_{\rho} :=
  \FF_q[\theta]\omega_{1} + \FF_q[\theta]\omega_{2}$ be the period
  lattice of $\rho$.  Suppose that $\lambda_{1}, \dots, \lambda_{m}\in
  \CC_{\infty}$ satisfy $\exp_{\rho}(\lambda_{i}) =: \alpha_{i}\in
  \ok$ for $i=1, \dots, m$ and that $\omega_{1}, \omega_{2},
  \lambda_{1}, \dots, \lambda_{m}$ are linearly independent over
  $k$. Finally, let $\{ \Gamma_{[n]}\}_{0\leq n \leq m}$ be defined as
  above. If $\Gamma_{[n-1]}$ has full dimension for some $1\leq n \leq
  m$, then so does $\Gamma_{[n]}$. In particular, the $4+2m$ elements
\[
\omega_{1}, \omega_{2}, \lambda_{1}, \dots, \lambda_{m},
F_{\tau}(\omega_{1}), F_{\tau}(\omega_{2}), F_{\tau}(\lambda_{1}),
\dots, F_{\tau}(\lambda_{m})
\]
are algebraically independent over $\ok$.
\end{lemma}

\begin{proof}
Since $\Gamma_{[0]} = \GL_2$, we see that $\Gamma_{[n-1]}$ has full
dimension when $n=1$.  Thus we assume for arbitrary $n$ that
$\Gamma_{[n-1]}$ has full dimension but that $\Gamma_{[n]}$ does
not.  Then by Lemma \ref{L:dimGamman} we have $\dim
\Gamma_{[n]}=\dim \Gamma_{[n-1]}$. Moreover, by
Lemma~\ref{L:GammanEqs} there exist $\ell_{11},$ $\ell_{21}$,
$\ell_{1}, \dots, \ell_{n}\in \FF_q(t)$ with $\ell_{n}\neq 0$ so
that
\begin{gather*}
   \phi_{1} := \ell_{11}X_{11} + \ell_{21}X_{21} + \ell_{1}X_{1} + \dots
+ \ell_{n}X_{n}-\ell_{11},   \\
   \phi_{2} : =\ell_{11}X_{12} + \ell_{21}X_{22} + \ell_{1}Y_{1} + \dots
+ \ell_{n}Y_{n}-\ell_{21},
\end{gather*}
are defining polynomials for $\Gamma_{[n]}$.

By Theorem \ref{T:GalThy}, $Z_{\Psi_{n}}$ is a torsor for
$\Gamma_{[n]}\times_{\FF_q(t)}\ok(t)$ over $\ok(t)$.  Since
$\Gamma_{[n]}$ is an affine linear space over $\FF_q(t)$,
$Z_{\Psi_{n}}$ is also an affine linear space over $\ok(t)$.  Hence
$Z_{\Psi_{n}}(\ok(t))$ is non-empty.  Choosing $\delta\in
Z_{\Psi_{n}}(\ok(t))$, we see that $Z_{\Psi_{n}}(\ok(t))=\delta \cdot
\Gamma_{\Psi_{n}}(\ok(t))$, which implies that
\[
AX_{11} + BX_{21} + \ell_{1}X_{1} + \cdots+ \ell_{n}X_{n}-\ell_{11},\quad
AX_{12} + BX_{22} + \ell_{1}Y_{1} + \cdots + \ell_{n}Y_{n}-\ell_{21},
\]
are defining polynomials for $Z_{\Psi_{n}}$, where
\[
[A,B,\ell_{1},\dots,\ell_{n}] :=
[\ell_{11},\ell_{21},\ell_{1},\dots,\ell_{n}]\delta^{-1}.
\]
We claim that $A$, $A^{(-1)}$, $B$, $B^{(-1)} \in \ok(t)$ are
regular at $t=\theta$. Let
$[F_{i},G_{i}]:=\mathbf{g}_{i}^{\mathrm{tr}}\Psi$ for $i=1,\dots,n$.
Then by the definition of $Z_\Psi$ in \S\ref{S:DiffGal} we have two
equations
\begin{gather}\label{E:ABFi1}
A\Psi_{11} + B\Psi_{21} + \ell_{1}F_{1} + \cdots + \ell_{n}F_{n}
- \ell_{11}=0,\\
\label{E:ABFi2}
A\Psi_{12} + B\Psi_{22} + \ell_{1}G_{1} + \cdots + \ell_{n}G_{n}
- \ell_{21}=0.
\end{gather}
Using the $\sigma$-action on \eqref{E:ABFi1} and then
subtracting it from itself, we have:
\[
\biggl( A-B^{(-1)}(t-\theta)-\sum_{i=1}^{n}
\alpha_{i}\ell_{i} \biggr)\Psi_{11}+\bigl(
B-A^{(-1)}+\kappa^{(-1)}B^{(-1)}\bigr) \Psi_{21} = 0.
\]
Since $\Gamma_{[0]}=\Gamma_{\Psi}=\GL_{2}$, Theorem \ref{T:GalThy}
implies that the four functions $\{ \Psi_{ij} : i=1,2,\ j=1,2\}$ are
algebraically independent over $\ok(t)$.  Hence
\begin{equation}\label{E:ABtwist}
A - B^{(-1)}(t-\theta) - \sum_{i=1}^{n} \alpha_{i}\ell_{i} =0, \quad
B - A^{(-1)} + \kappa^{(-1)}B^{(-1)} = 0,
\end{equation}
and thus,
\[
B+\kappa ^{(-1)} B^{(-1)}-(t-\theta^{(-1)} ) B^{(-2)}=\sum_{i=1}^{n}
\alpha_{i}^{(-1)}\ell_{i}.
\]
Note that the right hand side of this equation is regular at
$t=\theta^{q^{i}}$ for all $i\in \ZZ$. Indeed if $B$ has a pole at
$t=\theta$, then either $B^{(-1)}$ or $B^{(-2)}$ has pole at
$t=\theta$.  That is, either $B$ has pole at $t=\theta^{q}$ or $B$ has
pole at $t=\theta^{q^{2}}$.  Continuing this argument, we see that $B$
has infinitely many poles among $\{t =
\theta^{q^{i}}\}_{i=1}^{\infty}$, which contradicts that $B\in \ok(t)$.
Using a similar argument, we see that $B^{(-1)}$ is also regular at
$t=\theta$. By \eqref{E:ABtwist}, we thus see that $A$ and $A^{(-1)}$
are regular at $t=\theta$ and that
\begin{equation}\label{E:Atheta}
A(\theta)=\sum_{i=1}^{n} \alpha_{i} \ell_{i}(\theta).
\end{equation}

Recall that $\Psi(\theta)$ is given explicitly in
\eqref{E:PsiRhoTheta}.  By specializing \eqref{E:ABFi1}
and \eqref{E:ABFi2} at $t=\theta$ and using
\eqref{E:Atheta}, we have
\begin{gather}\label{E:spec1}
\biggl( \sum_{i=1}^{n}\ell_{i}(\theta) \lambda_{i} \biggr)\xi
F_{\tau}(\omega_{2})+ \biggl(B(\theta)-\sum_{i=1}^{n}
\ell_{i}(\theta)F_{\tau}(\lambda_{i}) \biggr)\xi
\omega_{2} - \ell_{11}(\theta)\tpi=0, \\
\label{E:spec2}
-\biggl( \sum_{i=1}^{n}\ell_{i}(\theta)\lambda_{i}\biggr)\xi
F_{\tau}(\omega_{1}) - \biggl(B(\theta)-\sum_{i=1}^{n}
\ell_{i}(\theta)F_{\tau}(\lambda_{i}) \biggr)\xi
\omega_{1}-\ell_{21}(\theta)\tpi = 0.
\end{gather}
Using the analogue of Legandre's relation \eqref{E:Legendre}, then
$\eqref{E:spec1} \times \omega_{1}+\eqref{E:spec2} \times \omega_{2}=0$ implies
\[
\ell_{1}(\theta)\lambda_{1}+\cdots+\ell_{n}(\theta)\lambda_{n}
-\ell_{11}(\theta)\omega_{1}-\ell_{21}(\theta)\omega_{2}=0.
\]
Since $\ell_{n}(\theta)\neq 0$, we obtain a non-trivial $k$-linear
dependence among $\omega_1$, $\omega_2$, $\lambda_1, \dots,
\lambda_n$, which is a contradiction.
\end{proof}

\begin{theorem}\label{T:LogMain}
  Suppose that $p$ is odd.  Let $\rho$ be any rank $2$ Drinfeld
  $\FF_q[t]$-module defined over $\ok$ without complex
  multiplication. Let $\lambda_{1}, \dots, \lambda_{m} \in \CC_\infty$
  such that $\exp_{\rho}(\lambda_{i})\in \ok$ for
  $i=1, \dots, m$.  If $\lambda_{1}, \dots, \lambda_{m}$ are linearly
  independent over $k$, then the $2m$ elements
\[
\lambda_{1}, \dots, \lambda_{m}, F_{\tau}(\lambda_{1}),
\dots,F_{\tau}(\lambda_{m})
\]
are algebraically independent over $\ok$.
\end{theorem}

\begin{proof}
By Remark \ref{R:GammaMisGL2} we may assume without loss of
generality that the coefficient of $\tau^2$ in $\rho_t$ is $1$. Let
$\Lambda_{\rho}:=\FF_q[\theta]\omega_{1} + \FF_q[\theta]\omega_{2}$
be the period lattice of $\rho$. Let $\mathbf{\Lambda}_{\rho}$ be
the $k$-vector space spanned by $\omega_{1}$ and $\omega_{2}$, and
let $\{ v_{1},v_{2} \}$ be any $k$-basis of
$\mathbf{\Lambda}_{\rho}$.  Given any $z_{1}, \dots, z_{n} \in
\CC_{\infty}$ such that $\exp_{\rho}(z_{i})\in \ok$ for
$i=1,\dots,n$, we observe that
\begin{gather}\label{E:omegazkspan}
k\hbox{-Span}\{\omega_{1},\omega_{2}, z_{1},\dots,z_{n}
\} = k\hbox{-Span}\{v_{1},v_2,z_{1},\dots,z_{n} \};\\
\label{E:okomegazFtau}
\ok\bigl(\cup_{j=1}^{n}
\cup_{i=1}^{2}\{\omega_{i},F_{\tau}(\omega_{i}),z_{j},F_{\tau}(z_{j})
\}\bigr)= \ok\bigl(\cup_{j=1}^{n}
\cup_{i=1}^{2}\{v_{i},F_{\tau}(v_{i}),z_{j},F_{\tau}(z_{j})
\}\bigr).
\end{gather}
Note that \eqref{E:okomegazFtau} follows from the
fact that the quasi-periodic function $F_{\tau}$ is
$\FF_q[\theta]$-linear on $\Lambda_{\rho}$ and satisfies difference
equations as in \eqref{E:QuasiPerFunEq}.

We define $N := k\hbox{-Span}\{\omega_{1},\omega_{2}, \lambda_{1},
\dots, \lambda_{m}\}$.  By Lemma \ref{L:redII}, the theorem is
immediately true if $\dim_{k} N = m+2$.  If $\dim_{k}N = m+1$, then
without loss of generality we can assume that $\omega_2 = b_0
\omega_1 + \sum_{j=1}^m b_j \lambda_j$, $b_j \in k$, with $b_1 \neq
0$.  In that case, $\{ \omega_1, \sum b_j\lambda_j \}$ is a
$k$-basis of $\mathbf{\Lambda}_\rho$ and $\{ \omega_1, \sum
b_j\lambda_j, \lambda_2, \dots \lambda_m \}$ is a $k$-basis of $N$.
The result then follows from Lemma \ref{L:redII} combined with
\eqref{E:omegazkspan} and \eqref{E:okomegazFtau}.  If $\dim_{k} N =
m$, then in a similar manner, we can find $b_j$, $c_j \in k$ so that
$\{ \sum_j b_j\lambda_j ,\sum_j c_j\lambda_j \}$ is a $k$-basis of
$\mathbf{\Lambda}_{\rho}$ and $\{ \sum_j b_j\lambda_j ,\sum_j
c_j\lambda_j, \lambda_3, \dots, \lambda_m\}$ is a $k$-basis of $N$.
Again the result follows from Lemma \ref{L:redII} combined with
\eqref{E:omegazkspan} and \eqref{E:okomegazFtau}.
\end{proof}

\bibliographystyle{amsplain}

\end{document}